\documentclass[]{siamart190516}
\usepackage[margin=1.25in]{geometry} 
\usepackage{amsfonts,amssymb}
\usepackage{graphicx}
\usepackage{subcaption}
\usepackage{xfrac}
\usepackage{mathabx}
\usepackage{enumitem}
\usepackage{physics}
\usepackage{enumitem}
\usepackage{bbm}

\mathchardef\ordinarycolon\mathcode`\:
\mathcode`\:=\string"8000
\begingroup \catcode`\:=\active
\gdef:{\mathrel{\mathop\ordinarycolon}}
\endgroup
\usepackage{tikz}
\usetikzlibrary{calc,arrows}

\newcommand{\ts}{\text{  }}

\usepackage{dsfont}

\usepackage{hyperref}
\hypersetup{
	colorlinks=true,
	linkcolor=blue,
	filecolor=magenta,      
	urlcolor=cyan,
}

\title{A Numerical Scheme for Wave Turbulence: 3-Wave Kinetic Equations\thanks{The authors are  funded in part by  the  NSF RTG Grant DMS-1840260, NSF Grants DMS-1854453, DMS-2204795, DMS-2305523,    Humboldt Fellowship,   NSF CAREER  DMS-2044626/DMS-2303146.}}
\author{Steven Walton\thanks{Department of Mathematics, Southern Methodist University, Dallas, Texas 75275, USA.  Current address: XCP-4, Los Alamos National Laboratory, Los Alamos, New Mexico 87544, USA} \and Minh-Binh Tran   
\thanks{Department of Mathematics, 
	Texas A\&M University, College Station, Texas 77843, USA}    
 }
\date{}

\begin{document}
	
	\maketitle
	\begin{center}{ \small {``Dedicated to the 70th Birthday of Professor Duong Minh Duc''}}\end{center}

	\begin{abstract}
		We introduce a finite volume scheme to solve a special case of isotropic 3-wave kinetic equations. We test our numerical solution against theoretical results concerning the long time behavior of the energy and observe that our solutions verify the energy cascade phenomenon. To our knowledge, this is the first numerical scheme that can capture the long time asymptotic behavior of solutions to those isotropic 3-wave kinetic equations, where the energy cascade can be observed.  Our numerical energy cascade rates are in good agreement with previously obtained theoretical results. The finite volume  scheme given here relies on a new identity, allowing one to  reduce the number of terms needed in the collision operators.

	\end{abstract}
	
	
	\section{Introduction}
	For more than 60 years, the theory of weak wave turbulence has been intensively developed. Having origins in the works of Peierls \cite{Peierls:1993:BRK,Peierls:1960:QTS}, with modern developments originating in the works of Hasselman \cite{hasselmann1962non,hasselmann1974spectral}, Benney and Saffmann \cite{benney1966nonlinear}, Kadomtsev \cite{kadomtsev1965plasma}, Zakharov \cite{zakharov2012kolmogorov},  Benney and Newell \cite{benney1969random} , wave turbulence kinetic equations have been shown to play an important role in a vast range of physical applications. Well known  examples are  water surface gravity and capillary waves, internal waves on density stratification and inertial waves due to rotation in planetary
	atmospheres and oceans, waves on quantized vortex lines in super
	fluid helium and planetary Rossby waves in
	weather and climate evolution.

	In rigorously deriving wave turbulence kinetic equations, Lukkarinen and Spohn are pioneers with the work \cite{LukkarinenSpohn:WNS:2011}.  In the last few years, the rigorous justification of wave turbulence theory has been revisited in \cite{buckmaster2016effective,collot2020derivation,deng2021full,dymov2019formal,staffilani2021wave}, in an effort to tackle the long standing open conjecture on the longtime behavior of higher order Sobolev norms, $H^s$ $(s>1)$, of  solutions to  dispersive equations on the torus.

	The 3-wave kinetic equation, one of the most important classes of  wave kinetic  equations, reads (see \cite{pushkarev1996turbulence,pushkarev2000turbulence,zakharov1968stability,zakharov1967weak,ZakharovNazarenko:DOT:2005})
	\begin{equation}\label{WeakTurbulenceInitial}
		\begin{aligned}
			\partial_tf(t,p) \ =& \ \mathcal{Q}[f](t,p), \ 
			f(0,p) \ =& \ f_0(p),
		\end{aligned}
	\end{equation}
	in which  $f(t,p)$ is the nonnegative wave density  at  wavenumber $p\in \mathbb{R}^N$, $N \ge 2$; $f_0(p)$ is the initial condition.  The quantity $\mathcal{Q}[f]$ describes pure resonance and is of the form 
	\begin{equation}\label{def-Qf}\mathcal{Q}[f](p) \ = \ \iint_{\mathbb{R}^{2N}} \Big[\mathcal R_{p,p_1,p_2}[f] -\mathcal R_{p_1,p,p_2}[f] -\mathcal R_{p_2,p,p_1}[f] \Big] d^Np_1d^Np_2 \end{equation}
	with $$\begin{aligned}
		\mathcal	R_{p,p_1,p_2} [f]:=  |V_{p,p_1,p_2}|^2\delta(p-p_1-p_2)\delta(\omega -\omega_{1}-\omega_{2})(f_1f_2-ff_1-ff_2) 
	\end{aligned}
	$$
	with the short-hand notation $f = f(t,p)$, $\omega = \omega(p)$ and $f_j = f(t,p_j),$ $\omega_j = \omega(p_j)$, for wavenumbers $p$, $p_j$, $j\in\{1,2\}$. The function $\omega(p)$ is the dispersion relation of the wave system. The 3-wave kinetic equation has a variety of applications from ocean waves, acoustic waves, gravity capillary waves to Bose-Einstein condensates and many others (see \cite{hasselmann1962non,hasselmann1974spectral,PomeauBinh,zakharov1965weak,zakharov1968stability,zakharov1967weak} and references therein).
	
	In the isotropic case, we identify  $f(t,p)$ with $f(t,\omega)$, the isotropic 3-wave kinetic equation, in which only the forward transferring part of the collision operator is considered, takes the form
	\begin{align}\label{EE1Colm}
		\begin{split}
			\partial_t f(t,\omega) \ =& \ {Q}[f](t,\omega), \ \ \ \ \omega\in\mathbb{R}_+,\
			f(0,p) \ = \ f_0(p),\\
			{{Q}}[f](t,\omega) \ =& \ \int_0^\infty\int_0^\infty \big[R(\omega, \omega_1, \omega_2)-R(\omega_1,\omega, \omega_2)-R(\omega_2, \omega_1, \omega) \big]{d}\omega_1{d}\omega_2, \\\
			& R(\omega, \omega_1, \omega_2):=  \delta (\omega-\omega_1-\omega_2)
			\left[ U(\omega_1,\omega_2)f_1f_2-U(\omega,\omega_1)ff_1-U(\omega,\omega_2)ff_2\right]\,,
		\end{split}
	\end{align}
	where $U$ satisfies $| U(\omega_1,\omega_2) |  \ = \ (\omega_1\omega_2)^{\gamma/2},$ in which $\gamma$ is a non-negative constant which plays an important role what follows. Note that the operator ${Q}[f](t,\omega)$ can be considered as a coagulation kinetic type operator for wave interactions.
	
	One of the breakthrough  results in weak turbulence theory ( see \cite{KorotkevichDyachenkoZakharov:2016:NSO,zakharov1967weak,zakharov2012kolmogorov}) concerns  the existence of the so-called  Kolmogorov-Zakharov (KZ) spectra, which is a class of {\it time-independent solutions} $f_\infty$ of equation \eqref{WeakTurbulenceInitial}:  $$f_\infty(p) \approx C|p|^{-\kappa}, \ \ \ \kappa>0.$$

 The KZ solutions are analogous to the  Kolmogorov energy spectrum 
	of hydrodynamic turbulence, though the value of $\kappa$ in the weak turbulence theory is dependent upon the wave system under consideration.  Research in line with this topic has actively continued to the present (see, for instance \cite{galtier2000weak,lvov2010oceanic,micha2004turbulent}, to name only a few). However, in absence of forcing and dissipation, the KZ solution scaling is only expected for infinite capacity systems, e.g. for a forward cascade process such systems that the energy integral diverges at infinite $p$. In the opposite case of the finite capacity systems, the spectrum blows up in a finite time.  The systems we consider in the present article are of the latter class.
	
	To our knowledge, relatively little has been done on the time-dependent solutions of \eqref{WeakTurbulenceInitial}. In the  important works \cite{connaughton2009numerical,connaughton2010aggregation,connaughton2010dynamical}, several numerical experiments were designed to  investigate the isotropic 3-wave equation. In \cite{connaughton2009numerical}, it was pointed out that isotropic 3-wave kinetic equations are equivalent to mean field rate equations for an aggregation-fragmentation problem which possesses an unusual fragmentation mechanism.  A numerical method for solving isotropic 3-wave kinetic equations, with forcing and dissipation present, was also introduced in the same work.

	In \cite{soffer2019energy}, Soffer and Tran show that, the energy conserved isotropic solutions of \eqref{WeakTurbulenceInitial}, in the finite capacity case (with $\gamma>1$), exhibit the   property that the energy is cascaded from small wavenumbers to large wavenumbers. They  show that for a regular initial condition whose energy at infinity, $\omega=\infty$,  is initially $0$, as time evolves, the energy is gradually accumulated at $\{\omega=\infty\}$. In the long time limit, all the energy of the system is concentrated at $\{\omega=\infty\}$ and the energy function becomes a Dirac function at infinity $\mathcal E\delta_{\{\omega=\infty\}}$, where $\mathcal E$ is the total energy. To be more precise, let us define the energy  of the solution  \eqref{WeakTurbulenceInitial} as $
g(t,\omega) \ = \ \omega f(t,\omega)$. It has been proved in \cite{soffer2019energy} that $g$ can be decomposed into two parts
\begin{equation}
	\label{Decomposition} 
	g(t,\omega) \ = \  \bar{g}(t,\omega) \ + \ \tilde{g}(t)\delta_{\{\omega=\infty\}},
\end{equation}
where $\bar{g}(t,\omega)\ge 0$ is the regular part, which is a  function, and $\tilde{g}(t)\delta_{\{\omega=\infty\}}$, is the singular part, which is a measure. The function $\tilde{g}(t)$ is non-negative. Initially, $\bar{g}(0,\omega)={g}(0,\omega)$ and $\tilde{g}(0)=0$. But, there exists a blow-up time $t^*_1$, such that for all time $t>t^*_1$, the function $\tilde{g}(t)$ is strictly positive. Moreover, starting from time $t^*_1$, there exists {\it infinitely many blow-up times} \begin{equation}
	\label{Decomposition1} 0< t^*_1<t^*_2<\cdots<t_n^*<\cdots,\end{equation}
such that \begin{equation}
	\label{Decomposition2}  \bar{g}(t_1^*,\omega)>\bar{g}(t_2^*,\omega)>\cdots>\bar{g}(t_n^*,\omega)> \cdots \to 0,\end{equation}
and  \begin{equation}
	\label{Decomposition3} 0<\tilde{g}(t_1^*)<\tilde{g}(t_2^*)<\cdots<\tilde{g}(t_n^*)< \cdots\end{equation}
The phenomenon has been explained in \cite{soffer2019energy} that, after the first blow-up time, $t^*_1$,
the energy starts to transfer from the regular part $\bar{g}(t,\omega)$ to the singular part $\tilde{g}(t)\delta_{\omega=\infty}$ at a rate at least like $\mathcal{O}(\frac{1}{\sqrt{t}})$, while the total energy of the two regular and singular parts is still conserved. This decay rate has been obtained in item (iii) of the main theorem of \cite{soffer2019energy} (Theorem 10, pages 22-45), which states as follows.  The energy cascade has an explicit rate
$\int_{\{|p|=\infty\}}f(t,|p|) \omega_{|p|}|p|^2d\mu(|p|)\ \ge \ \mathfrak{C}_1 \ - \  \frac{\mathfrak{C}_2}{\sqrt{t}},$
where $\mathfrak{C}_1$ and $\mathfrak{C}_2$ are explicit constants.
 Item (iii) of the main theorem states that the total energy is accumulated at the point $\{\infty\}$ with the rate $\mathcal{O}(\frac{1}{\sqrt{t}})$.
The above inequality yields with $\omega = |p|$ the equivalent
\begin{equation*}\int_{0}^R g(t,\omega)d\omega\ \	\le \   \mathcal{O}(\frac{1}{\sqrt{t}}),\mbox{ for any $R>0.$}
 \end{equation*}

In the limit that $t\to\infty$, all of the energy will be accumulated to the singular part $\tilde{g}(t)\delta_{\omega=\infty}$, while the regular part will vanish, $\bar{g}(t,\omega)\to0$. This means if we look for a strong  (non-measured) solution, whose energy is conserved, it can only exist up to a very short time $t=t^*_1,$ at which point it exhibits singular behavior. As a result, we refer to time $t^*_1$ as the {first blow-up} time.
  Let $\chi_{[0,R]}(\omega)$ be a cut-off function of $\omega$ on the finite domain $[0,R]$, the {\it multiple blow-up time } phenomenon \eqref{Decomposition}-\eqref{Decomposition1}-\eqref{Decomposition2}-\eqref{Decomposition3}, with the decay rate $\mathcal{O}(\frac{1}{\sqrt{t}})$, can be observed equivalently as the decay of the total energy on any finite interval $[0,R]$
 \begin{equation}
 	\label{Decomposition4}
 \int_{0}^R g(t,\omega)d\omega\ = \	\int_{\mathbb{R}_+}\chi_{[0,R]}(\omega)   g(t,\omega)d\omega\le \  \mathcal{O}\Big(\frac{1}{\sqrt{t}}\Big) \mbox{   as  } t\to\infty, \mbox{ for all truncated parameter } R.
 \end{equation}
Inequality \eqref{Decomposition4} simply means that the energy of the solution will move away from any truncated finite interval $[0,R]$ as $t\to\infty$ with the rate $\mathcal{O}\Big(\frac{1}{\sqrt{t}}\Big).$

We refer to \cite{soffer2019energy} for a detailed comparison between the different results of \cite{connaughton2009numerical,connaughton2010aggregation,connaughton2010dynamical,soffer2019energy}. Below, a  brief comparison will be given. In \cite{connaughton2009numerical,connaughton2010aggregation,connaughton2010dynamical}, both infinite capacity ($0\le \gamma\le1$) and finite capacity cases ($\gamma>1$) have been considered.  
The solutions of \cite{connaughton2010aggregation,connaughton2010dynamical} are assumed to follow a self-similar hypothesis, called { dynamic scaling}
\begin{equation}\label{Ansartz}
	f(t,\omega)\approx s(t)^a F\left(\frac{\omega}{s(t)}\right)
\end{equation}
where $\approx$ denotes the scaling limit $s(t)\to\infty$ and $\omega\to \infty$ with $x=\omega/s(t)$ fixed. The energy of this function can be computed as follows
\begin{equation}\label{EnergyColm}
	\begin{aligned}
		\int_{0}^\infty \omega f(t,\omega)d\omega\ = & \ \int_0^\infty s(t)^a F\left(\frac{\omega}{s(t)}\right)\omega d\omega \ =  \ \int_0^\infty s(t)^{a+2} F\left(\frac{\omega}{s(t)}\right)\left(\frac{\omega}{s(t)}\right) d\left(\frac{\omega}{s(t)}\right)\\
		\ = & \ s(t)^{a+2} \int_0^\infty x F\left(x\right) dx \backsim \mathcal{O}\Big(s(t)^{a+2}\Big),
	\end{aligned}
\end{equation}
that grows with the rate $s(t)^{a+2}$.
Substituting this ansatz into the equation \eqref{EE1Colm}, we obtain the system
\begin{equation}\label{Eigen}
	\begin{aligned}
		& \dot{s}(t) \ =  \ s^\zeta, \mbox{ with } \zeta = \gamma + a + 2\ \mbox{ and }  aF(x) \ + \ x\dot{F}(x) \ =  \ {{Q}}[F](x).
	\end{aligned}
\end{equation}
From \eqref{EnergyColm}, it can be observed  that the only value of $a$ that gives the conservation of energy is $a=-2$. Making the assumption that $F(x)\backsim x^{-n}$, when $x\backsim 0$, this work shows that the power can be determined to be $n=\gamma+1$. Since \cite{connaughton2010aggregation} focuses on the infinite capacity case, the degree of homogeneity $\gamma$ is considered in the interval $[0,1)$, as thus the integral 
\begin{equation}\label{ColmEnergy3}
\int_0^\infty xF(x) dx
\end{equation} is well-defined. However, there is a problem in the finite capacity case: when $\gamma>1$, this integral becomes singular.

The work \cite{connaughton2010dynamical}, considers both infinite capacity ($0\le \gamma\le1$) and finite capacity cases ($\gamma>1$). In the finite capacity case, the solution is considered before the first blow-up time $t<t_1^*$, theoretically proved in \cite{soffer2019energy}. However,  solving \eqref{Eigen} is a very difficult task. In \cite{connaughton2010dynamical}, a hypothesis is then needed: the   total energy of the solution of \eqref{EE1Colm} is assumed to grow linearly in time rather than being conserved \begin{equation}\label{ColmEnergy2}
	\int_{0}^\infty \omega f(t,\omega)d\omega=Jt.
\end{equation}
Another  challenging technical issue is that the integration \eqref{ColmEnergy3} with $F(x)\backsim x^{{-\frac{\gamma+3}{2}}}$ for small $x$  diverges in the finite capacity case  ($\gamma>1$) and   converges only in the infinite capacity case $(0\le \gamma\le 1)$. As thus, other assumptions need to be imposed on the solution $f$ itself. 

{\it From the above discussions, the key differences between \cite{connaughton2009numerical,connaughton2010aggregation,connaughton2010dynamical} and \cite{soffer2019energy} are summarized as follows. The works complement each other as they consider very different scenarios of the solutions of \eqref{WeakTurbulenceInitial}.  The work \cite{soffer2019energy} focuses on the finite capacity case $(\gamma>1)$, under no additional assumption on the solution $f$ of \eqref{EE1Colm} and shows that the solution, whose energy is conserved rather than linearly grows in time as assumed in \eqref{ColmEnergy2}, exhibits an energy cascade phenomenon, where there exist an infinite series of blow up times \eqref{Decomposition}-\eqref{Decomposition1}-\eqref{Decomposition2}- \eqref{Decomposition3} (or, equivalently, inequality \eqref{Decomposition4}). The result of \cite{soffer2019energy} is in good agreement with the discussion of \cite{Nazarenko:2011:WT}, saying that in the absence of external forcing and dissipation, the KZ spectrum is not expected in the finite capacity   ($\gamma>1$) case. The works \cite{connaughton2009numerical,connaughton2010aggregation,connaughton2010dynamical} focus on both cases - finite capacity   ($\gamma>1$) and    infinite capacity  $(0\le \gamma\le 1)$. In the finite capacity case ($\gamma>1$),  the solution is studied before the first blow-up time $t<t_1^*$, rigorously proved later in  \cite{soffer2019energy}. Studying the solution before the first blow-up time via the self-similar hypothesis \eqref{Ansartz} is indeed a highly challenging problem, and, as thus, in  \cite{connaughton2010dynamical} an assumption needs to be imposed on  the energy of the solution: it is  assumed to grow linearly in time (see \eqref{ColmEnergy2}) rather than being conserved
\begin{equation}\label{ColmEnergy2a}
	\int_{0}^\infty \omega f(t,\omega)d\omega=\mbox{ constant}.
\end{equation}
 Moreover, additional hypotheses are also imposed to treat the singularities of the integral \eqref{ColmEnergy3}.  }

 We would like to highlight the work \cite{bell2017self}, where a self-similar profile of the solution for a  different finite capacity system  - the Alfven wave
 turbulence kinetic equation - is computed  before the first blow-up time $t_1^*$. Also, a recently published paper \cite{semisalov2021numerical} presents a numerical method for solving the  self-similar profile before the first blow-up time $t_1^*$  for a collision integral 4-wave kinetic equation based on Chebyshev approximations.    {\it While finding self-similar profiles of the solution $f$ of \eqref{WeakTurbulenceInitial}  not only before the first blow-up time $t_1^*$ but also before the $n$-th blow-up time $t_n^*$,  is a topic of our next  work, our current manuscript only focuses on numerically verifying the existence of the multiple-blow up time $t^*_1<t^*_2<\cdots<t_n^*<\cdots$ phenomenon, as well as the bound \eqref{Decomposition4}, rigorously proved in \cite{soffer2019energy}.
}
 
 Let us also mention the work \cite{RumpfSofferTran}, where  a 3-wave kinetic equation, derived from the elastic beam wave equation on the lattice, has been analyzed. It has been shown that the domain of integration of the 3-wave collision operator is broken into disconnected domains, each has their own local equilibrium. If one starts with any initial condition, whose energy is finite on one subdomain, the solutions will relax to the local equilibrium of this subregion, as time evolves. This is the so-called non-ergodicity phenomenon, which is different from the energy cascade phenomenon observed in our current work.  The global existence of 3-wave kinetic equations in the presence of forcing has been done in \cite{GambaSmithBinh,nguyen2017quantum} and a connection to chemical reaction networks has also been pointed out in \cite{CraciunSmithBoldyrevBinh}.

Starting from \eqref{EE1Colm}, we could rewrite the 3-wave kinetic equation under the following equivalent form 
\begin{equation}\label{Final3wave}
	\begin{aligned}
		\partial_t f(t,k) \ = \ \mathbb Q[f](t,k), \ \ \ \ k\in\mathbb{R}_+, \ 
		f(0,k) = f_0(k),
	\end{aligned}
\end{equation}
in which $\mathbb Q$ is the collision operator defined by
\begin{equation}\label{WKE_Collision_operator}
	\begin{aligned}
		\mathbb{{Q}}[f](t,k) \ =& \ \int_0^k [
		a(k_1,k-k_1)f(k_1)f(k-k_1)-a(k,k_1)f(k)f(k_1)\ 
		   -a(k,k-k_1)f(k)f(k-k_1)]\mathrm{d}k_1\\
		&\ - 2\int_0^\infty [
		a(k,k_1)f(k)f(k_1)-a(k+k_1,k_1)f(k+k_1)f(k_1)\ 
		   -a(k_1+k,k)f(k)f(k_1+k)]\mathrm{d}k_1,
	\end{aligned}
\end{equation}
where the collision kernel satisfies $ a(k_1,k_2)  \ = \ (k_1 k_2)^{\gamma/2}$. {\it In the rest of our paper, we identify the variable $k$ with the variable $\omega$ as $k$ is  simply $\omega$ multiplied by a fixed constant.}   

Even though the theoretical result of \cite{soffer2019energy} holds for the more general equation \eqref{WeakTurbulenceInitial} in the isotropic case, it is expected that the result also hold for equation \eqref{EE1Colm}, in which only the part that drives the forward cascade is kept. Therefore, the goal of our paper is then to derive a finite volume scheme that allows us to observe the time evolution of the solutions of \eqref{Final3wave} and to verify the theoretical results of \cite{soffer2019energy}  for various values of $\gamma>1$. In other words, we aim to observe the transferring of energy from the regular to the singular part in \eqref{Decomposition} and to measure precisely the rate of this energy transfer process  via the inequality \eqref{Decomposition4}, that we call the energy cascade rate. {\it In the absence of forcing and dissipation, the KZ spectrum is not expected in the finite capacity   ($\gamma>1$) case considered in our current work \cite{Nazarenko:2011:WT}. }

Our finite volume scheme relies on the combination of a new energy identity represented in Lemma \ref{lemma:identity} and an adaptation of  Filbet and Lauren\c cot's scheme \cite{Fil04}   for the Smoluchowski coagulation equation to the 3-wave kinetic equation. Most numerical schemes that approximate integrals on an unbounded domain require the truncation of the unbounded domain to a finite domain.  Thanks to the new identity, the number of terms in the collision operators is reduced, which reduces the number of truncations needed in the approximation, making the scheme more accurate and reliable. Indeed, we only need to truncate one term in our numerical scheme. Let us comment   our the CFL condition is  restrictive in certain cases. For other types of  equations, the issue of positivity and accurate long-time behavior could be  resolved with using the  implicit in time discretizations. For instance, a structure preserving scheme has been designed for the Kolmogorov Fokker Planck equation, which is an equation of  degenerate parabolic type in the work \cite{FOSTER2017319}.

The advantage of degenerate parabolic type equations is that those equations are normally local while 
 the 3-wave kinetic equation considered in the current manuscript is highly non-local. As thus, the previous strategies  for parabolic equations, such that the one used for the above Kolmogorov-Fokker-Planck equation, do not immediately carry over to the current 3-wave kinetic equation. We therefore refer this question as a topic of  our forthcoming paper.


	\section{Comparison with Smoluchowski Coagulation Equation and a New Energy Identity } 
	In \cite{Fil04}, Filbet and Lauren\c cot  derive a finite volume scheme (FVS) for the Smoluchowski coagulation equation (SCE)
	\begin{equation}
		\begin{aligned}
			\partial_t f(t,k) \ = \ \mathbb Q_{Smo}[f](t,k), \ 
			f(0,k) = f_0(k),
		\end{aligned}
	\end{equation}	
	\begin{equation}\label{Smol_collision_operator}
		\begin{aligned}
			\mathbb{{Q}}_{Smo}[f](t,k) \ =& \ \int_0^k 
			a(k_1,k-k_1)f(k_1)f(k-k_1)\mathrm{d}k_1\ - 2\int_0^\infty 
			a(k,k_1)f(k)f(k_1)\mathrm{d}k_1,
		\end{aligned}
	\end{equation}
	where $a(\cdot, \cdot)$ is the collision kernel for the 3-wave collision operator (\ref{WKE_Collision_operator}).  Let us give a short derivation of the non-conservative form of the SCE (see \cite{Bak91}, \cite{daC98}, \cite{Fil04} and the references therein).
	\par 
	Take a test function  $\phi(k)=k\chi_{[0,c]}(k)$ and apply it to the SCE
	$$\int_0^c\partial_t f(t,k)k\mathrm{d}k \ = \ \int_0^c\int_0^k 
	{a(k,k-k_1)f(k_1)f(k-k_1)k}\mathrm{d}k_1\mathrm{d}k\ - 2\int_0^c\int_0^\infty 
	{a(k,k_1)f(k)f(k_1)}k\mathrm{d}k_1\mathrm{d}k$$
	
	$$=2\int_0^c\int_0^{c-k} 
	{a(k,k_1)f(k)f(k_1)k}\mathrm{d}k_1\mathrm{d}k\ - 2\int_0^c\int_0^\infty 
	{a(k,k_1)f(k)f(k_1)}k\mathrm{d}k_1\mathrm{d}k.$$
	Rearranging the right-hand side, we find 
	$$\int_0^c\partial_t f(t,k){k}\mathrm{d}k \ = \ -2\int_0^c\int_{c-k}^\infty 
	{a(k,{k}_1)f(k_1)f(k)k}\mathrm{d}k_1\mathrm{d}k,$$
	and upon taking the derivative with respect to $c$, 
	$$\partial_t f(t,c){c} \ = \ -2\partial_c \int_0^c\int_{c-k}^\infty 
	{a(k,{k}_1)f(k_1)f(k)k}\mathrm{d}k_1\mathrm{d}k,$$
	and, after truncating the inner integral, we arrive at
	\begin{equation}\label{NC_FL}
		\partial_tf(t,c){c} \ = \ -2\partial_c \int_0^c\int_{c-k}^{R} 
		{a(k,{k}_1)f(k_1)f(k)k}\mathrm{d}k_1\mathrm{d}k,
	\end{equation}
	where $R$ is a suitable truncation of the volume domain.  Then, one can apply any finite volume scheme to solve the truncated problem. We note briefly that the choice of $R$ can effect the accuracy and efficiency of the scheme for the SCE.  For example, when considering the case of gelation, $R$ must be quite large in order to avoid a loss of mass before the gelation time.  Also, in \cite{Bor13}, the authors compare the FVS above with a finite element approximation and find that for smaller truncation values, the finite element scheme is a better choice, while for larger truncation values the FVS should be used. 
	\par
	If we compare equation (\ref{WKE_Collision_operator}) with equation (\ref{Smol_collision_operator}), we see that the Smoluchowski coagulation equation is a special case of the 3-wave equation.   Thus, we would like to adapt the (FVS) of Filbet and Lauren\c cot to derive a numerical scheme for the 3-wave equation.  To do this we derive a similar identity to (\ref{NC_FL}) for the wave kinetic equation. The role of this identity is to reduce the number of terms in the collision operator. As discussed previously, the truncation of the terms in the collision operator can affect the accuracy of the numerical scheme, thus reducing the number of truncated terms is crucial in the numerical computations of the solutions, as it allows the scheme to be more accurate and reliable.

	\par 	However, for the 3-wave equation the wave density, $f(t,k)$, is not conserved, but the energy, $g(t,k)=kf(t,k)$, is conserved and so we solve for the energy.  What's more, since we do not have an analytic solution to test against, we validate our scheme by verifying  the energy cascade rate for the 3-wave equation found in Soffer and Tran \cite{soffer2019energy} and discussed above.    
	
	\begin{lemma}\label{lemma:identity}
		The following identity holds true for the energy function $g(t,k)$
		
		\begin{equation}\label{3_scheme_g}
			\begin{aligned}
				\partial_t\frac{ g(t,c)}{c}\ =  & \ -2\partial_c\int_0^c \int_0^c a(k,k_1)\frac{g(k)}{k}\frac{g(k_1)}{k_1}\chi \big\{c < k + k_1 \big\}\mathrm{d}k_1\mathrm{d}k\\
				& \ \ \ \	+ \partial_c\int_0^\infty\int_0^\infty a(k,k_1)\frac{g(k)}{k}\frac{g(k_1)}{ k_1}\chi \big\{c < k + k_1 \big\}\mathrm{d}k\mathrm{d}k_1,
			\end{aligned}
		\end{equation}
		with $\chi\big\{ \cdot \big\}$ the characteristic function and initial condition $g(0,k) = g_0(k) = kf(0,k)$.  
		
	\end{lemma}
	\begin{proof}

		Let $\phi(k)$ be a test function. From \cite{soffer2019energy}, we have the following identity for the 3-wave equation
		\begin{equation}\label{WKE_identity}
			\int_0^\infty\partial_t f(t,k)\phi(k)\mathrm{d}k \ = \ \int_0^\infty\int_0^\infty a(k,k_1)f(k)f(k_1)[\phi(k+k_1)
			+\phi(|k-k_1|)	-2\phi(\max\{k,k_1\})]\mathrm{d}k\mathrm{d}k_1.
		\end{equation}
		If we choose $\phi(k)=\chi_{[0,c]}(k)$ (compare this with $ \phi(k)=k\chi_{[0,c]}(k)$ in \cite{Fil04}), we have
		\begin{equation}\label{test_id}
			\int_0^\infty\partial_tf(t,k)\chi_{[0,c]}(k)\mathrm{d}k \ = \ \int_0^\infty\int_0^\infty a(k,k_1)f(k)f(k_1)[\chi_{[0,c]}(k+k_1)+\chi_{[0,c]}(|k-k_1|)-2\chi_{[0,c]}(\max\{k,k_1\})]\mathrm{d}k\mathrm{d}k_1.
		\end{equation}
		Set $K(k,k_1) = \chi_{[0,c]}(k+k_1)+\chi_{[0,c]}(|k-k_1|)-2\chi_{[0,c]}(\max\{k,k_1\})$.  There are seven cases:
		\begin{enumerate}[label=\roman*]
			\item Assume $k + k_1 \leq c\ts$ then we have $\chi_{[0,c]}(k + k_1) = 1$, $\chi_{[0,c]}(|k-k_1|) = 1$ and $\chi_{[0,c]}(\max\{k,k_1\}) = 1$ which implies that $K(k,k_1) = 0$.

			\item  Assume that $k + k_1 > c$ but $k\leq c$ and $k_1\leq c$ Then we have $\chi_{[0,c]}(k + k_1) = 0$, $\chi_{[0,c]}(|k-k_1|) = 1$, $\chi_{[0,c]}(\max\{k,k_1\}) = 1$ and so $K(k,k_1) = -1$.

			\item  Now let $k + k_1 > c$, $k\leq c$ but  $k_1 > c$ then $\chi_{[0,c]}(k + k_1) = 0$, $\chi_{[0,c]}(|k-k_1|) = 1$, and $\chi_{[0,c]}(\max\{k,k_1\}) = 0$ giving $K(k,k_1) = 1$.

			\item Assume $k + k_1 > c$, with $k > c$ and $k_1\leq c$ and $|k-k_1| > c$, then   $\chi_{[0,c]}(k + k_1) = 0$, $\chi_{[0,c]}(|k-k_1|) = 0$, $\chi_{[0,c]}(\max\{k,k_1\}) = 0$ leaving $K(k,k_1) = 0$.
			
			\item Assume $k + k_1 > c$, with $k > c$ and $k_1\leq c$ and $|k-k_1| \leq c$, then   $\chi_{[0,c]}(k + k_1) = 0$, $\chi_{[0,c]}(|k-k_1|) = 1$, $\chi_{[0,c]}(\max\{k,k_1\}) = 0$ leaving $K(k,k_1) = 1$.
			
			\item  Lastly, set $k + k_1 > c$, $k > c$,  $k_1 > c$, and $|k-k_1|>c$, then $\chi_{[0,c]}(k + k_1) = 0 $, $ \chi_{[0,c]}(|k-k_1|) = 0$ and $\chi_{[0,c]}(\max\{k,k_1\}) = 0$ resulting in $K(k,k_1) = 0$.

			\item  Lastly, set $k + k_1 > c$, $k > c$,  $k_1 > c$, and $|k-k_1|\leq c$, then $\chi_{[0,c]}(k + k_1) = 0 $, $ \chi_{[0,c]}(|k-k_1|) = 1$ and $\chi_{[0,c]}(\max\{k,k_1\}) = 0$ resulting in $K(k,k_1) = 1$.
		\end{enumerate}

		Applying these computations to (\ref{test_id}) and  after taking the derivative as done for (\ref{NC_FL}), we have
		
		\begin{equation}
			\begin{aligned}
				\partial_t f(t,c) \ = & \ -\partial_c\int_0^\infty\int_0^\infty a(k,k_1)f(k)f(k_1)\chi\big\{ c < k+k_1\big\}\chi\big\{k\wedge k_1 \leq c \big\}\mathrm{d}k_1\mathrm{d}k\\
				&\ \ \ \ 	+ \partial_c\int_0^\infty\int_0^\infty a(k,k_1)f(k)f(k_1)\chi\big\{c < k+k_1\big\}\chi\big\{k\wedge k_1 > c \big\}\mathrm{d}k\mathrm{d}k_1,
			\end{aligned}
		\end{equation}
		which gives,
		
		\begin{equation}\label{3_scheme_f}
			\begin{aligned}
				\partial_t f(t,c)\ =  & \ -2\partial_c\int_0^c \int_0^c a(k,k_1)f(k)f(k_1)\chi \big\{c < k + k_1 \big\}\mathrm{d}k_1\mathrm{d}k\\
				& \ \ \ \	+ \partial_c\int_0^\infty\int_0^\infty a(k,k_1)f(k)f(k_1) k_1\chi \big\{c < k + k_1 \big\}\mathrm{d}k\mathrm{d}k_1,
			\end{aligned}
		\end{equation}
		
		yielding \eqref{3_scheme_g}.
	\end{proof}
	
	With the identity \eqref{3_scheme_g} in hand, we now derive a finite volume scheme with which to solve it. Note that with \eqref{3_scheme_g}, we only need to truncate a single term, while with the original formula
	\eqref{WKE_Collision_operator}, we are  required to truncate three terms.

	\section{Finite Volume Scheme}\label{scheme_section}
	
	We give a discretization for the frequency domain for $k\in [0,R]$. Let $i\in \{1, 2, \ldots , M \} = I^M_h$, with the maximum stepsize $h \in (0,1)$ fixed.  We define the set of cells
	$\mathcal{K} = \bigcup_{i\in I^M_h} K_i = [0,R]$
	to be the discretization of the wavenumber domain into $M$ cells.  Let
	\[
	\begin{aligned}
		K_i = [k_{i-1/2}, k_{i+1/2})_{i\in I^M_h}, && \{k_i\}_{i\in I^M_h} = \frac{k_{i+1/2} +  k_{i-1/2}}{2}, && \{h_i\}_{i\in I^M_h} = k_{i+1/2} - k_{i-1/2},
	\end{aligned}
	\] 
	define the cells, pivots and step-size respectively, with the boundary nodes $k_{1/2} = 0$ and $k_{M+1/2} = R$.  Note that the discretization does not require a uniform grid, but does restrict the step-size in each cell $K_i$ so that $h_i \leq h$.  In practice, it can be useful to drop the restriction that $h\in(0,1)$, especially when using a non-uniform grid, though we leave it here to simplify the analysis.  The set $T_N = \{0,\ldots,T\} $ with $N+1$ nodes, where $T$ is the maximum time, is the discretization of the time domain.  We fix the time step to be $\Delta t = \frac{T}{N}$, and denote by $t_n = \Delta t\cdot n$ for $n\ \in \{0,\ldots,N\}$.  We approximate equation (\ref{3_scheme_g}) with 
	\begin{equation}\label{scheme}
		g^{n+1}(k_{i}) = g^n(k_{i})  +\lambda_i\Big( Q^n_{i+1/2}\Big[\frac{g}{k} \Big] - Q^n_{i-1/2}\Big[\frac{g}{k} \Big] \Big),
	\end{equation}
	where $\lambda_i = \frac{k_{i}\Delta t }{h_i}$, and 
	\[
	Q^n_{i+1/2}\Big[\frac{g}{k} \Big] - Q^n_{i-1/2}\Big[\frac{g}{k} \Big] = -2\Big(Q^n_{1,i+1/2}\Big[\frac{g}{k} \Big]-Q^n_{1,i-1/2}\Big[\frac{g}{k} \Big]\Big) + \Big(Q^n_{2,i+1/2}\Big[\frac{g}{k} \Big]-Q^n_{2,i-1/2}\Big[\frac{g}{k} \Big]\Big),
	\]
	with
	\begin{equation}\label{J}
		Q^n_{1,i+1/2}\Big[\frac{g}{k} \Big] = \sum^{i}_{m=1}h_m \frac{g^n(k_m)}{k_m} \Bigg(\sum^i_{j=1}h_j \frac{g^n(k_j)}{k_j}a(k_m, k_j)\chi\Big\{k_{i+1/2} < k_m +k_j \Big\}
		\Bigg),
	\end{equation}
	
	\begin{equation}\label{C}
		Q^n_{2,i+1/2}\Big[\frac{g}{k} \Big] =  \sum^M_{m=1}h_m \frac{g^n(k_m)}{k_m} \Bigg(\sum^M_{j=1}h_j \frac{g^n(k_j)}{k_j}a(k_m, k_j)\chi\Big\{k_{i+1/2} < k_m+k_j\Big\}\Bigg),
	\end{equation}
	where we have used the midpoint rule to approximate the integrals in equation (\ref{3_scheme_g}) and we choose an explicit time stepping method. 
	\par
	We will be interested in computing the moments of our solution. We approximate the $\ell$-th moment, $\mathcal{M}^\ell(t_n)$, by 
	\begin{equation}\label{moments}
		\mathcal{M}^\ell(t_n) = \sum^M_{i=1} h_i g^n(k_i) k^\ell_i,
	\end{equation}
	with $\ell\in\mathbb{N}$.
	\par
	The initial condition $g_0(k)$ is approximated by
	\[
	g^0(k_{i}) = \frac{1}{h_i}\int^{k_{i+1/2}}_{k_{i-1/2}}g_0(k) \mathrm{d}k \approx g_0(k_i),
	\]
	by again employing the midpoint rule.  We have that $g^0(k_i) \geq 0$ for all $i\in\{1,\ldots,M \}$ since $g_0(k) \geq 0$ for all $k\in [0, R]$ by assumption.  To ease notation in the proof of the following proposition, we write $g^n_i$ for $g^n(k_{i})$ and the negative flux as
	\begin{equation}\label{eqn::neg_flux}
	 \Big[Q^n_i\Big]^{\mp} =  Q^n_{i-1/2}\Big[\frac{g}{k} \Big]  - Q^n_{i+1/2}\Big[\frac{g}{k} \Big], 
	\end{equation}
	 so that the forward Euler scheme becomes 
    $
      g^{n+1}_i = g^n_i - \lambda_i \Big[ Q_i^n \Big]^{\mp}.
    $
	Let us analyze the collision operator \eqref{eqn::neg_flux}.  For the moment, we leave \eqref{eqn::neg_flux} in continous form but will apply the same midpoint approximation when appropriate.  Then we have 
    \begin{equation}
     \begin{aligned}
         \Big[ Q_i^n \Big]^{\mp} = \Bigg(2\int^{k_{i+1/2}}_0 \int^{k_{i+1/2}}_0 g^n_1 g^n_2 (k_1k_2)^{\gamma/2-1} \chi\Big\{k_{i+1/2} < k_1 + k_1 \Big\}\mathrm{d}k_2\mathrm{d}k_1 \\
         - \int^{R}_0 \int^{R}_0 g^n_1 g^n_2 (k_1k_2)^{\gamma/2-1}\chi\Big\{k_{i+1/2} < k_1 + k_1 \Big\}\mathrm{d}k_2\mathrm{d}k_1 \Bigg)\\
         - \Bigg(2\int^{k_{i-1/2}}_0 \int^{k_{i-1/2}}_0 g^n_1 g^n_2 (k_1k_2)^{\gamma/2-1} \chi\Big\{k_{i-1/2} < k_1 + k_1 \Big\}\mathrm{d}k_2\mathrm{d}k_1 \\
         - \int^{R}_0 \int^{R}_0 g^n_1 g^n_2 (k_1k_2)^{\gamma/2-1} \chi\Big\{k_{i-1/2} < k_1 + k_1 \Big\}\mathrm{d}k_2\mathrm{d}k_1 \Bigg)\
         = (\mathrm{I} - \mathrm{II})- (\mathrm{III} - \mathrm{IV}),
     \end{aligned}
    \end{equation}
    with $g^n_i=g^n(k_i)$ for $i = 1,2$.  We may decompose the integrals above so that 
    
\begin{equation}
     \begin{aligned}
         \mathrm{I} = 2\int^{k_{i-1/2}}_0\int^{k_{i-1/2}}_0 g^n_1 g^n_2 (k_1k_2)^{\gamma/2-1} \chi\Big\{k_{i+1/2} < k_1 + k_1 \Big\}\mathrm{d}k_2\mathrm{d}k_1  \\
         +\, 2\int^{k_{i-1/2}}_0\int^{k_{i+1/2}}_{k_{i-1/2}}g^n_1 g^n_2 (k_1k_2)^{\gamma/2-1} \chi\Big\{k_{i+1/2} < k_1 + k_1 \Big\}\mathrm{d}k_2\mathrm{d}k_1 \\
         + \, 2\int^{k_{i+1/2}}_{k_{i-1/2}}\int^{k_{i-1/2}}_0 g^n_1 g^n_2 (k_1k_2)^{\gamma/2-1} \chi\Big\{k_{i+1/2} < k_1 + k_1 \Big\}\mathrm{d}k_2\mathrm{d}k_1 \\ 
         + \, 2\int^{k_{i+1/2}}_{k_{i-1/2}}\int^{k_{i+1/2}}_{k_{i-1/2}}g^n_1 g^n_2 (k_1k_2)^{\gamma/2-1} \chi\Big\{k_{i+1/2} < k_1 + k_1 \Big\}\mathrm{d}k_2\mathrm{d}k_1  \\
         = 2\int^{k_{i-1/2}}_0\int^{k_{i-1/2}}_0 g^n_1 g^n_2 (k_1k_2)^{\gamma/2-1} \chi\Big\{k_{i+1/2} < k_1 + k_1 \Big\}\mathrm{d}k_2\mathrm{d}k_1 \\
         + \, 4\int^{k_{i+1/2}}_{k_{i-1/2}}\int^{k_{i-1/2}}_0 g^n_1 g^n_2 (k_1k_2)^{\gamma/2-1} \chi\Big\{k_{i+1/2} < k_1 + k_1 \Big\}\mathrm{d}k_2\mathrm{d}k_1 \\
         + \, 2\int^{k_{i+1/2}}_{k_{i-1/2}}\int^{k_{i+1/2}}_{k_{i-1/2}}g^n_1 g^n_2 (k_1k_2)^{\gamma/2-1} \chi\Big\{k_{i+1/2} < k_1 + k_1 \Big\}\mathrm{d}k_2\mathrm{d}k_1 ,
     \end{aligned}
    \end{equation}
    and similar decompositions may be performed on II, III and IV, which, leaving out the $g^n_1 g^n_2 (k_1k_2)^{\gamma/2-1}$ integrands to simplify the expressions, gives 
     \begin{equation}
     \begin{aligned}
         \Big[ Q_i^n \Big]^{\mp} = 2\int^{k_{i-1/2}}_0\int^{k_{i-1/2}}_0  \Big(\chi\Big\{k_{i+1/2} < k_1 + k_1 \Big\}-\chi\Big\{k_{i-1/2} < k_1 + k_1 \Big\}\Big)\mathrm{d}k_2\mathrm{d}k_1 \\
         + \, \int^{R}_0\int^{R}_0  \Big(\chi\Big\{k_{i-1/2} < k_1 + k_1 \Big\}-\chi\Big\{k_{i+1/2} < k_1 + k_1 \Big\}\Big)\mathrm{d}k_2\mathrm{d}k_1 \\
         + \, 4 \int^{k_{i+1/2}}_{k_{i-1/2}}\int^{k_{i-1/2}}_0 \chi\Big\{k_{i+1/2} < k_1 + k_1 \Big\}\mathrm{d}k_2\mathrm{d}k_1 \ 
         + \,  2 \int^{k_{i+1/2}}_{k_{i-1/2}}\int^{k_{i+1/2}}_{k_{i-1/2}} \chi\Big\{k_{i+1/2} < k_1 + k_1 \Big\}\mathrm{d}k_2\mathrm{d}k_1, \\
         = -2 \int^{k_{i-1/2}}_0\int^{k_{i-1/2}}_0 \chi\Big\{k_{i-1/2} < k_1 + k_1 \leq k_{i+1/2} \Big\} \mathrm{d}k_2\mathrm{d}k_1 \\
         + \, \int^{R}_0\int^{R}_0 \chi\Big\{k_{i-1/2} < k_1 + k_1 \leq k_{i+1/2} \Big\} \mathrm{d}k_2\mathrm{d}k_1 \\
          + \, 4 \int^{k_{i+1/2}}_{k_{i-1/2}}\int^{k_{i-1/2}}_0 \chi\Big\{k_{i+1/2} < k_1 + k_1 \Big\}\mathrm{d}k_2\mathrm{d}k_1 \ 
         + \  2 \int^{k_{i+1/2}}_{k_{i-1/2}}\int^{k_{i+1/2}}_{k_{i-1/2}} \chi\Big\{k_{i+1/2} < k_1 + k_1 \Big\}\mathrm{d}k_2\mathrm{d}k_1.
     \end{aligned}
    \end{equation}
    At this point, we apply the midpoint rule to approximate the integrals which, after making use the characteristic functions, simplifies the above to 
    \begin{gather}
     \begin{aligned}
         \Big[ Q_i^n \Big]^{\mp} \approx 
         2 g^n_i k_i^{\gamma/2-1}h _i\Big( \sum_{j = 1}^i g_j^n k_j^{\gamma/2-1}h_j
         +  \sum_{j = 1}^{i-1} g_j^n k_j^{\gamma/2-1}h_j \Big)
         - \, \sum^{i-1}_{j=1} g^n_j g^n_{i-j} (k_j k_{i-j})^{\gamma/2-1}h_j h_{i-j}.
     \end{aligned}
    \end{gather}
    In what follow, we will abuse notation slightly and right $\Big[ Q_i^n \Big]^{\mp}$ when we mean its midpoint approximation above.

    We are now in a position to state the following sufficient stability condition for the time step.
    \begin{proposition}\label{propo:positive}
     If the time step, $\Delta t$ satisfies 
     \begin{equation}\label{eqn::cfl}
        \Delta t R^{\gamma+1} \| g^0 \|_{L^\infty(0,R)} \leq \frac{\gamma}{16}\min_{i\in I^M_h}{h_i}, 
     \end{equation}
     then, for all $n \geq 0$ and $i\in I^M_h$ we have
     $ 
      g^n_i \geq 0,
     $
     and further,
     $$ \| g^{n+1}\|_{L^\infty(0,R)} \leq C(\gamma)\| g^n \|_{L^\infty(0,R)},
     $$
     with $C(\gamma)\in[\frac{15}{16},1]$ so that finally,
     \begin{equation}
         \| g^{n}\|_{L^\infty(0,R)} \leq \| g^0 \|_{L^\infty(0,R)},
     \end{equation}
     for all $n \geq 0$.  
    \end{proposition}
    
    \begin{proof}
     Assume $R>>1$.  We will proceed by induction.  Starting from the forward Euler scheme and using $[Q^0_i]^\mp$ as computed above we obtain
    \begin{gather}
     \begin{aligned}
         g^{1}_i = g^0_i\Bigg( 1/2 - \lambda_i \Big( 2k^{\gamma/2-1}_ih_i \Big[ \sum^{i}_{j=1}g_j^0 k_j^{\gamma/2-1}h_j
         + \, \sum^{i-1}_{j=1}g_j^0 k_j^{\gamma/2-1}h_j \Big] \Big) \Bigg)\\
         + \, \frac{1}{2}g^0_i + \lambda_i\sum^{i-1}_{j=1}g^0_j g^0_{i-j}(k_j k_{i-j})^{\gamma/2-1}h_j h_{i-j}\ 
         = \varphi^0_i(g) + \zeta_i^0(g).
     \end{aligned}  
    \end{gather}
    Given $g^0_i \geq 0$, $\zeta^0_i(g)$ is clearly positive.  Then, to conclude that $g^1_i \geq 0$ it is sufficient to have $\varphi^0_i(g) \geq 0$ or equivalently 
    \begin{equation}
        \Delta t \frac{k_i}{h_i} \Big( 2k^{\gamma/2-1}_ih_i \Big[ \sum^{i}_{j=1}g_j^0 k_j^{\gamma/2-1}h_j
         + \, \sum^{i-1}_{j=1}g_j^0 k_j^{\gamma/2-1}h_j \Big] \Big) \leq \frac{1}{2}.
    \end{equation}
    To this end, note that 
    \begin{equation}
    \begin{aligned}
        {k_i} \Big( 2k^{\gamma/2-1}_ih_i \Big[ \sum^{i}_{j=1}g_j^0 k_j^{\gamma/2-1}h_j
         + \, \sum^{i-1}_{j=1}g_j^0 k_j^{\gamma/2-1}h_j \Big] \Big)\ 
         \leq 4k^{\gamma/2}_i\| g^0 \|_{L^\infty(0,R)}\|k^{\gamma/2-1}\|_{L^1(0,R)}\ 
         \leq \frac{8}{\gamma} R^{\gamma+1}\| g^0\|_{L^\infty(0,R)}.
    \end{aligned}
    \end{equation}
    For now, we provide a provisional stability condition on $\Delta t$ so that
    \begin{equation}\label{eqn::provisional_cfl}
        \Delta t R^{\gamma+1}\max_{\ell \leq n}\{\| g^\ell \|_{L^\infty(0,R)}\} \leq \frac{\gamma}{16}\min_{i\in I^M_h}{h_i},
    \end{equation}
    which will be seen to be equivalent to the condition \eqref{eqn::cfl} by the end of the proof.  
    Thus, with \eqref{eqn::provisional_cfl} setting $n=0$ we have that $\varphi^0_i(g)$ is positive.  Now, assuming $g^{n}_i \geq 0$, we proceed in the same fashion.  The steps are identical except that now we have the condition that 
    \begin{equation}\label{eqn::temp_cfl}
        \Delta t R^{\gamma+1}\| g^n \|_{L^\infty(0,R)} \leq \frac{\gamma}{16}\min_{i\in I^M_h}{h_i},
    \end{equation}
    which again is satisfied by the provisional assumption on $\Delta t$ and we have by induction that $g^{n+1}_i \geq 0$ as desired.  

    To show stability, let us consider $\zeta^n_i(g)$ and gives estimates for its second term.  To this end, note that
    \begin{equation}
        \begin{aligned}
            \frac{k_i}{h_i}\sum^{i-1}_{j=1}g^n_j g^n_{i-j}(k_j k_{i-j})^{\gamma/2-1}h_j h_{i-j}\ 
            \leq k_i \| g^n \|^2_{L^\infty(0,R)} \sum^M_{j=1}k_j^{\gamma-1}h_j\\
            \leq  \frac{k_i}{\gamma}R^\gamma \| g^n \|^2_{L^\infty(0,R)}\ 
            \leq \frac{1}{\gamma}R^{\gamma+1}\| g^n\|^2_{L^\infty(0,R)},
        \end{aligned}
    \end{equation}
    where the rearrangement inequality was used in the second line.  Then, using \eqref{eqn::provisional_cfl}, we obtain
    \begin{equation}
        \begin{aligned}
             \Delta t\frac{k_i}{h_i}\sum^{i-1}_{j=1}g^n_j g^n_{i-j}(k_j k_{i-j})^{\gamma/2-1}h_j h_{i-j} \leq \frac{1}{16}\| g^n \|_{L^\infty(0,R)} 
        \end{aligned}
    \end{equation}
    Once again using the assumption \eqref{eqn::provisional_cfl}, we see that 
    \begin{equation}
        \begin{aligned}
            g^{n+1}_i = \varphi^n_i(g) + \zeta^n_i(g) \leq g^n_i\Big(\frac{8-\gamma}{16}\Big) +\frac{1}{2}g^n_i + \frac{1}{16}\|g^n\|_{L^\infty(0,R)}, 
        \end{aligned}
    \end{equation}
    which implies that 
    \begin{equation}
        \| g^{n+1}\|_{L^\infty(0,R)} \leq \Big( \frac{17 -\gamma}{16} \Big) \| g^n\|_{L^\infty(0,R)} = C(\gamma)\| g^n\|_{L^\infty(0,R)}.
    \end{equation}
     We see that for $\gamma \in[1,2]$ we have $C(\gamma)\in [\frac{15}{16}, 1]$ and thus 
     \begin{equation}
         \| g^{n+1} \|_{L^\infty(0,R)} \leq \| g^0 \|_{L^\infty(0,R)},
     \end{equation}
     which gives the equivalence of condition \eqref{eqn::provisional_cfl} with \eqref{eqn::cfl}.  
    \end{proof}
    In the following section we provide numerical results for several initial conditions, which give evidence for the dependence of the time steps on the initial condition and truncation parameter as demonstrated above.

	\section{Numerical Tests}\label{tests}
	As mentioned above, the KZ spectrum is not expected in our equation, in absence of forcing and dissipation. Our numerical tests are therefore only designed to verify the blow up phenomenon first proved in \cite{soffer2019energy} as well as to confirm the energy cascade growth rate bound $\mathcal{O}(1/\sqrt{t})$ obtained in the same paper.  To verify this, in any finite interval of the wavenumber domain, we should see the total energy $\mathcal{E}$, or equivalently the zeroth moment of $g(t,k)$, decaying like $\mathcal{O}(\frac{1}{\sqrt{t}})$ in the interval. As  \eqref{Decomposition}-\eqref{Decomposition1}-\eqref{Decomposition2}- \eqref{Decomposition3} is  equivalent with  \eqref{Decomposition4}, our numerical tests will focus on verifying  \eqref{Decomposition4}.

	All tests were performed on a  uniform grid, with $h = 0.5$ in the first two test and $h = 0.1$ in the last two.  We solve the 3-wave equation via (\ref{3_scheme_g}) for four different initial conditions.  For each initial condition, we either vary the truncation parameter, $R$, in \eqref{Decomposition4}, and hold the degree of the collision kernel, $\gamma$, fixed or we vary $\gamma$ with $R$ kept at a fixed value. Specifically, for the first two initial conditions, we run tests for $R = 50, 100, 200$ with $\gamma =2$ fixed or $\gamma = \frac{3}{2}, \frac{9}{5}, 2$ with $R=100$ fixed.  In the last two cases, we do something similar and again run tests with $\gamma = 2$ fixed but set $R = 25, 50, 80$ and then fix $R=50$ and let $\gamma = \frac{3}{2}, \frac{9}{5}, 2$. We provide an approximation of the convergence rate for the first initial condition by comparison of approximated solutions over succesively refined grids.  In all solutions presented, we use a second-order Runge-Kutta scheme to integrate in time. 
	\par 
	The first few moments of the energy, $g(t,k)$, are then computed for each set of parameters.  We find our numerical experiments to be in good agreement with \cite{soffer2019energy}.  We note that, as with explicit methods for the Smoluchowski equation, the CFL condition can be very restrictive if one wants to maintain positivity (see \cite{Lai21}).  
	Similarly to schemes for other types of kinetic equations, maintaining positivity of the numerical solutions is also  an important issue in our scheme. 
	
	We provide a condition in Proposition \ref{propo:positive} that guarantees the positivity of the solutions. Under this condition, the positivity can be preserved when we choose $\Delta t$ sufficiently small.
	
	We observe the choice of $\Delta t$ depends on the initial data and the size of the truncated interval.  For example, in Test 1 we set $\Delta t = 0.05$ and in Test 2 we set $\Delta t = 0.005$, though we perform computations with the same truncation parameter.  Also, in Tests 3 and 4, as we increase the truncation parameter from $R=50$ to $R=80$, we must drop the time step from $\Delta t = 0.0004$ to $\Delta t = 0.00025$, respectively.  
	\par
	
	We believe that common positive preserving techniques  like those developed in \cite{hu2013positivity,huang2019positivity,huang2019third} could potentially be a solution to handle this instability issue.

	\par
	
	
	\subsection{Test 1}\label{test 1}
	
	Here we choose our initial condition to be
	\begin{equation}\label{spike}
		g_0(k) = 1.26157e^{-50(k-1.5)^2}
	\end{equation}
	with $\Delta t = 0.05$ for $t\in[0,T]$, $T= 10000$ ,  over a uniform grid, as mentioned previously, with $h = 0.5$.  The initial condition and final state, $g(T,k)$, are plotted in Figure \ref{g1_init}  and \ref{g1_final}, respectively, for $\gamma = 2$.  The theory developed in \cite{soffer2019energy}, in which it has been proven that as $t$ tends to $\infty$, the energy $g(t,k)$ converges to a delta function $\mathcal{E}\delta_{\{k=\infty\}}$, applies for $\gamma > 1$. Moreover, due to \eqref{Decomposition4}, the energy on any finite interval also goes to $0$ \begin{equation}\label{Cascade}
		\lim_{t\to\infty}g(t,k)\chi_{[0,R]}(k)=0.
	\end{equation}
	The two  Figures \ref{g1_init}  and \ref{g1_final} indeed  give strong evidence for the theoretical result
	\eqref{Cascade}, proved in \cite{soffer2019energy}.
	
\begin{figure}
\centering
\begin{subfigure}{.45\textwidth}
\includegraphics[width=\textwidth]{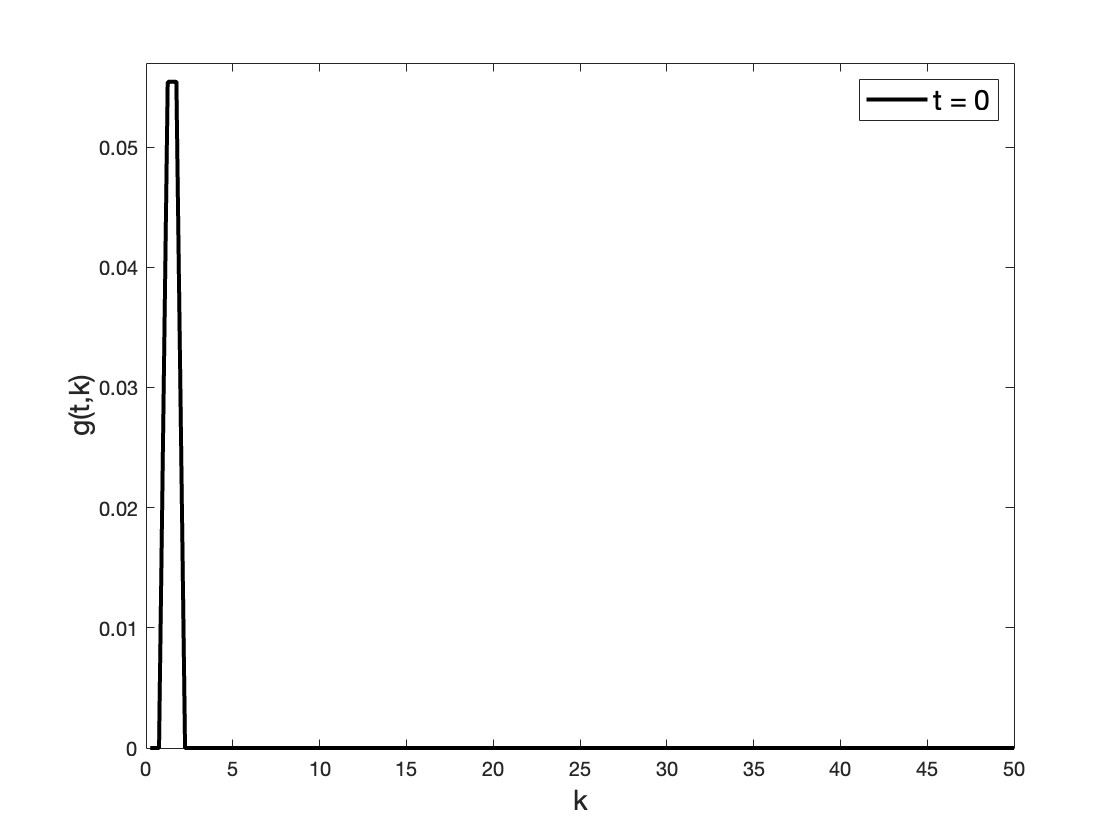}
\caption{\small Initial Condition \eqref{spike}.}\label{g1_init}
\end{subfigure}
\begin{subfigure}{.45\textwidth}
\includegraphics[width=\textwidth]{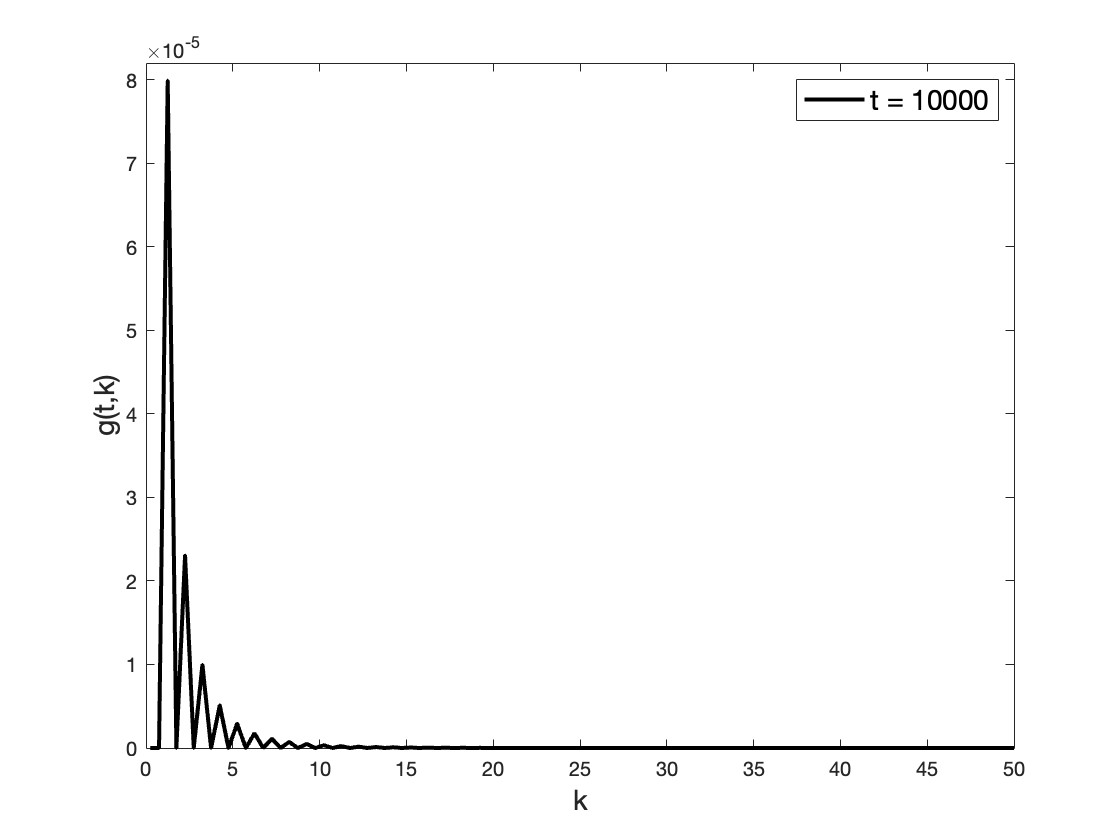}
\caption{\small Final Condition.}\label{g1_final}
\end{subfigure}
\caption{\small Test Case 1}
\label{g1}
\end{figure}

	\par
	
	\begin{figure}
		\begin{center}
			\includegraphics[scale = 0.3]{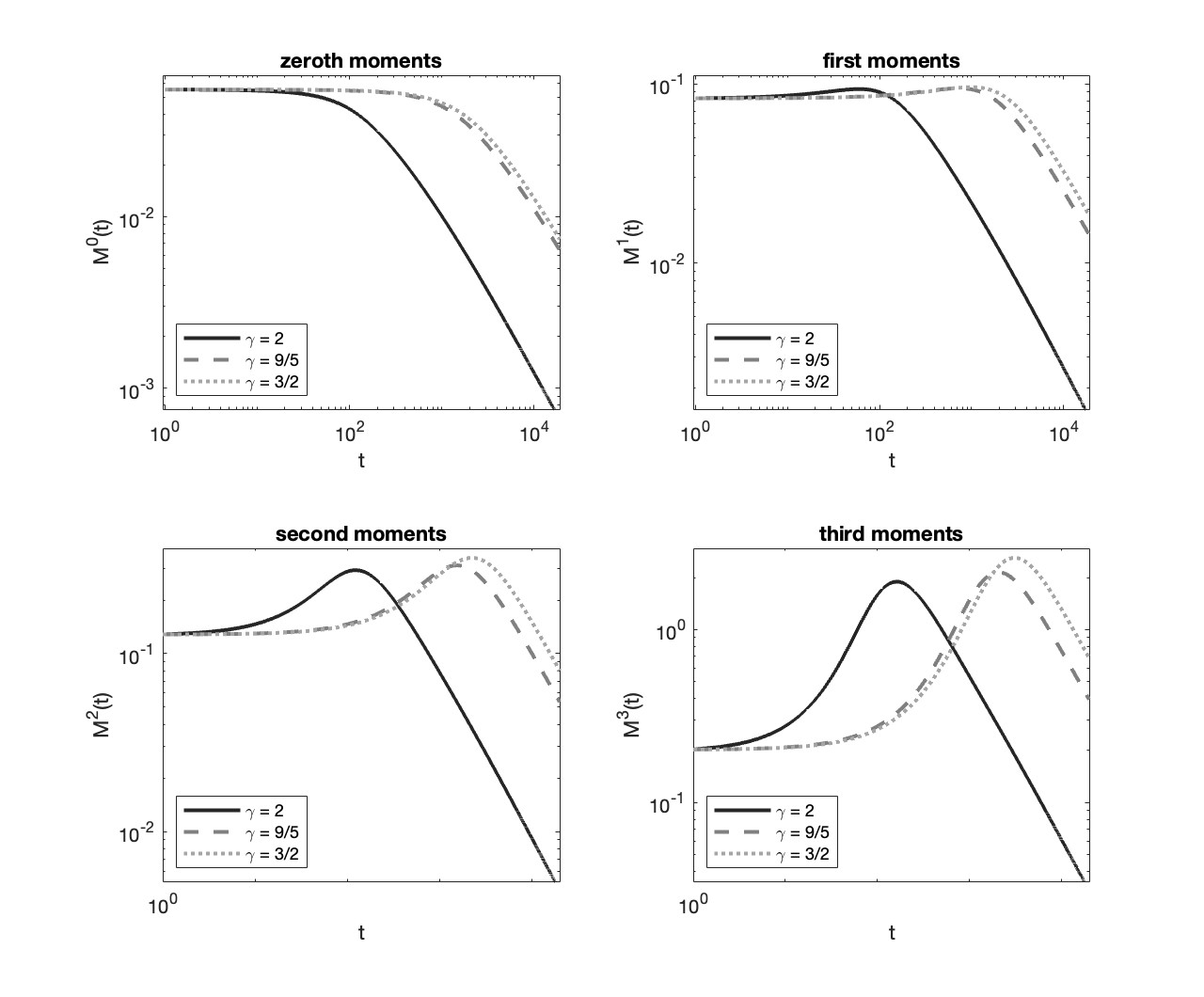}
			\caption{\small Here, we fix $R=100$ and compute the moments of the energy as a function of time for $\gamma = 3/2, 9/5\text{ and } 2$ for initial condition (\ref{spike}).}
			\label{g1_moments_gamma}
		\end{center}
	\end{figure}
	\par
	\par

	The first four moments of the energy are shown in Figure \ref{g1_moments_gamma} for $\gamma =\frac{3}{2}, \frac{9}{5}, 2$ and $k \in [0, 100]$. The numerical results in Figure \ref{g1_moments_gamma}  show that higher moments of $g$ also decay to $0$, due to \eqref{Cascade}.  We also see that the onset of decay happens later for smaller degree, $\gamma$.  For the zeroth moment, the curve corresponding to $\gamma=2$ is below the curve for $\gamma=3/2$ and the curve for $\gamma=9/5$. However, for the higher order moments, an interesting phenomenon happens.  The $\gamma = 3/2$ curve has a big jump when $t$ is around $10$ to $40$, and it lies above the other moments from that time on, while the $\gamma = 2$ curve is still below the $\gamma = 9/5$ curve at longer times.   While \eqref{Cascade} can be used to predict what happens in the Figures \ref{g1}, Figure \ref{g1_moments_gamma} indeed needs a different theoretical explanation, which should be an interesting subject of analysis in a follow-up paper. 
	\par
	
	We then investigate the moments for increasing values of the truncation parameter.  The first four moments of $g(t,k)$ are plotted in Figure \ref{g1_moments_R} with $\gamma=2$ fixed.
	For this initial condition, as might be expected, increasing the truncation parameter seems to have a negligible affect on the  higher moments.  For the other initial conditions in the test cases that follow, the difference is more distinguishable.  When we are able to make a distinction, the decay \eqref{Cascade} seems to be slower for larger values of the truncation parameter $R$, though the rate of cascade is unaffected. Moreover, when the initial condition spreads the energy across the interval, larger truncation parameters result in larger amounts of energy initially, as one might expect, but again the cascade rate is independent of the truncation value in these cases.  This observation is consistent with the findings of \cite{soffer2019energy} and can be understood as follows. As the energy moves away from the zero frequency $k=0$ and goes to the frequency $k=\infty$ as time evolves, the chance of having some energy in a finite interval $k\in[0,R]$ increases when $R$ increases. 
	\begin{figure}
		\begin{center}
	\includegraphics[scale = 0.3]{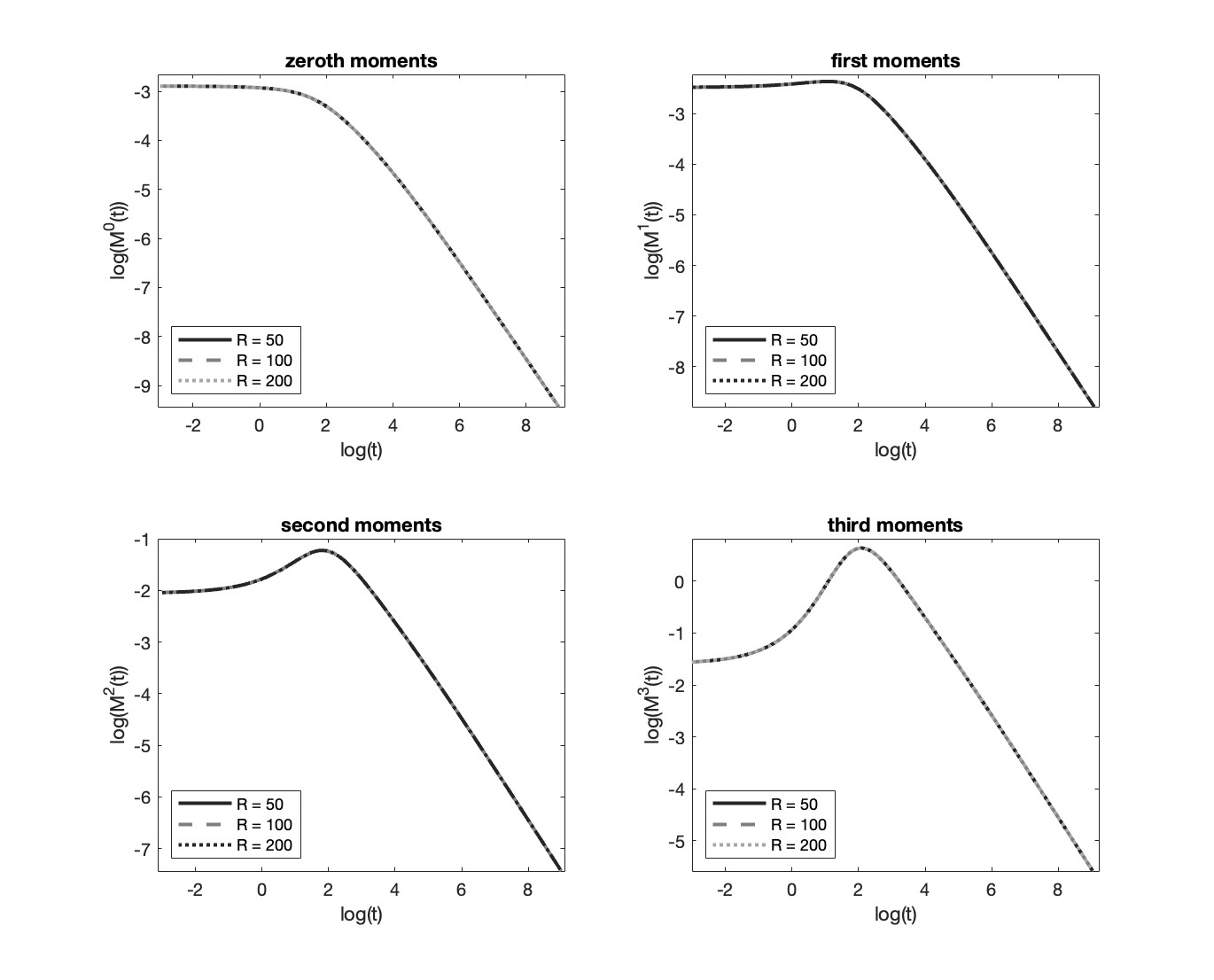}
			\caption{\small  For $\gamma = 2$ fixed, we compute the moments of the energy with initial condition (\ref{spike}) for truncation parameters $R=50, 100 ,200$.  }
			\label{g1_moments_R}
		\end{center}
	\end{figure}
    \par

 To test convergence without an exact solution, we provide two experimental order of convergence (EOC) tests.  First, we compare step sizes $h$, $h/2$ and $h/4$ for $R=50$ and $\gamma = 2$.  Here we run the simulation for $t\in[0,T_{EOC}]$ with $T_{EOC} = 25$ and $\Delta t=0.0125$ for every choice of $h$.  The approximate order of convergence, $p$, is given by \cite{leveque}
    \begin{equation}\label{EOC}
       p \approx \log_2 \Big( \frac{ \|g_h - g_{h/4}\|_{L^1}}{\|g_{h/2} - g_{h/4}\|_{L^1}} - 1 \Big) .
   \end{equation}
  The coarser solutions are interpolated using matlab's built-in piecewise cubic hermite interpolating polynomial (pchip) interpolator.  Table \ref{tbl::table} summarizes the approximation of $p$ for each $h$ at $t_{max} = \arg\max \|g_h - g_{h/4} \|_{L^1}$ .  

    \begin{table}[!h]
    \centering
       \begin{tabular}{ |c|c|c|c| } 
          \hline
          $h$ & 0.4 & 0.3 & 0.2 \\ 
          \hline
          $p$ & 2.5777 & 2.6071 & 2.6392\\ 
          \hline
        \end{tabular}
        \caption{Approximation of $p$ with formula \ref{EOC} for various $h$ values.}
        \label{tbl::table}
    \end{table}
    \par
    Next, we approximate $p$ by comparison with a fine grid solution, $g_{h^*}$, with $h^* = \frac{1}{80}$ and $\Delta t$ as before.  Then we approximate $p$ with the ratio \cite{leveque}
    \begin{equation}
        R_h(t) = \Big( \frac{ \|g_h - g_{h^*}\|_{L^1(0,R)}}{\|g_{h/2} - g_{h^*}\|_{L^1(0,R)}} \Big) .
    \end{equation}
    Then, the EOC is given by
    $ 
        p \approx \log_2 \{ R_h(T_{EOC}) \}.     
$    The results for $h=0.2$, $h=0.1$ are summarized in Table \ref{tbl::table_fine}.
    \begin{table}[!h]
    \centering
       \begin{tabular}{ |c|c|c| } 
          \hline
          $h$ & 0.2 & 0.1\\ 
          \hline
          $p$ & 2.0118 & 2.5288\\ 
          \hline
        \end{tabular}
        \caption{Approximation of $p$ with formula \ref{EOC} for various $h$ values.}
        \label{tbl::table_fine}
    \end{table}
    \par
    A log-log plot of the errors is provided in Figure \ref{fig::err_plots}.
    The verification of these experiments requires further analysis, which is the subject of a forthcoming  work.
    \begin{figure}
    \centering    
    \includegraphics[width=0.5\textwidth]{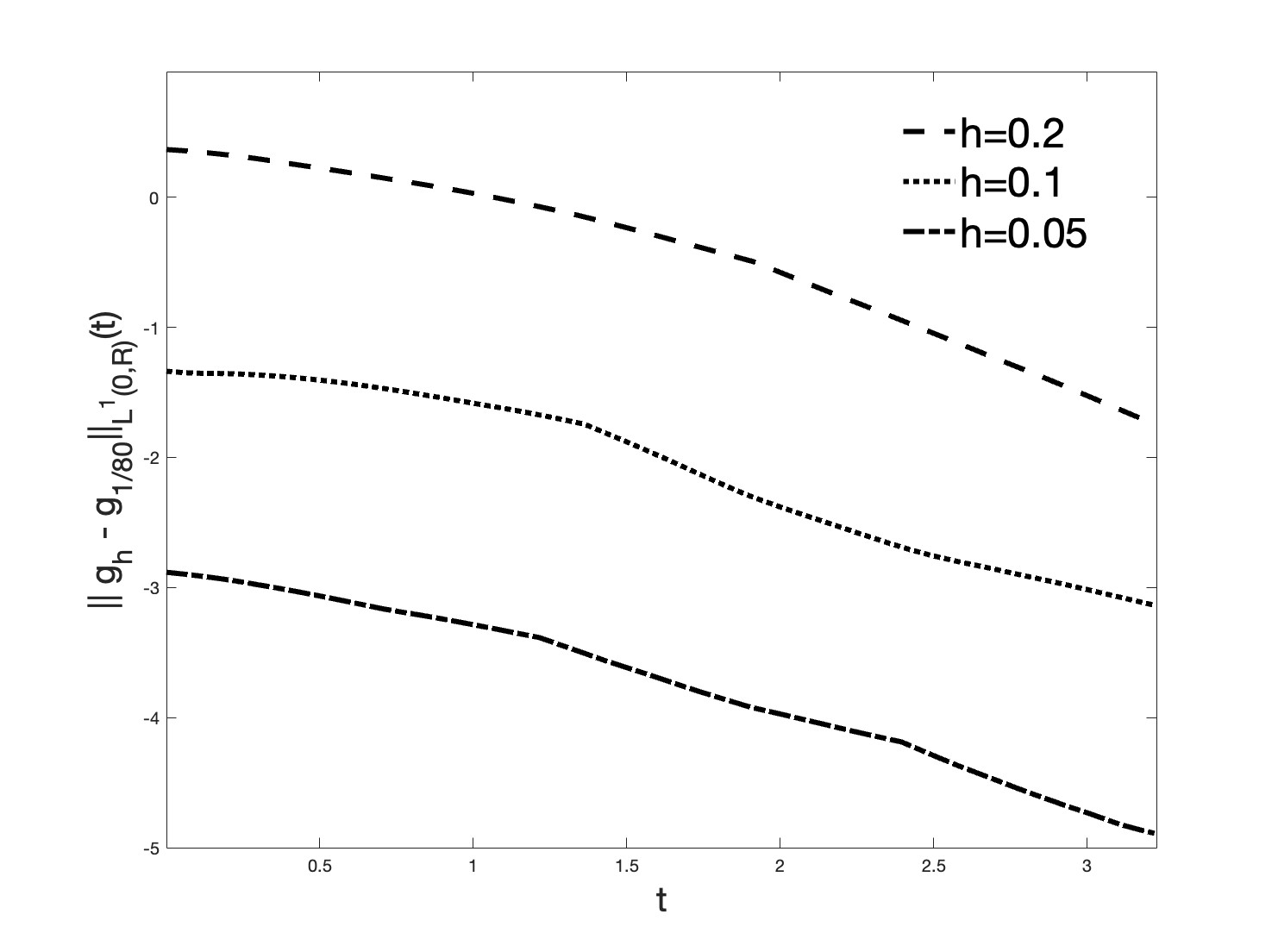}
    \caption{Log-log plot of the error between coarse and fine grid solutions.  }\label{fig::err_plots}
    \end{figure}

	\par
	
	We test the decay rate of the total energy in $[0,\infty)$ obtained in \cite{soffer2019energy} in Figures \ref{g1_decay_R} and \ref{rate_decay_test_1}.  We give a log-log plot of the zeroth moment of $g(t,k)$ for varying truncation parameter with $\gamma = 2$ and with fixed truncation parameter $R = 200$ with varying degree against the theoretical decay rate of the total energy in $[0,\infty)$.  In \cite{soffer2019energy}, it has been shown that the decay rate can be bounded by $\mathcal O\big(\frac{1}{\sqrt t}\big)$ (see \eqref{Decomposition4}), which has very good agreement with the numerical results, where the slope of the decay rate  curve in the long time limit are quite below the slope of the line corresponding to $\frac{1}{\sqrt{t}}$. We also compare with a rate like $\mathcal{O}(\frac{1}{t})$.  Here, and for the other initial conditions we consider, it would appear that the rate is like $\mathcal{O}(\frac{1}{t^s})$, with $s\in[1/2,1]$, depending on the degree and the initial condition.  Confirmation of this observation requires further analysis.      
	
		\begin{figure}[htpb]
		\centering
		\begin{subfigure}[b]{0.49\linewidth}	\includegraphics[width = \linewidth, height = 0.8\linewidth]{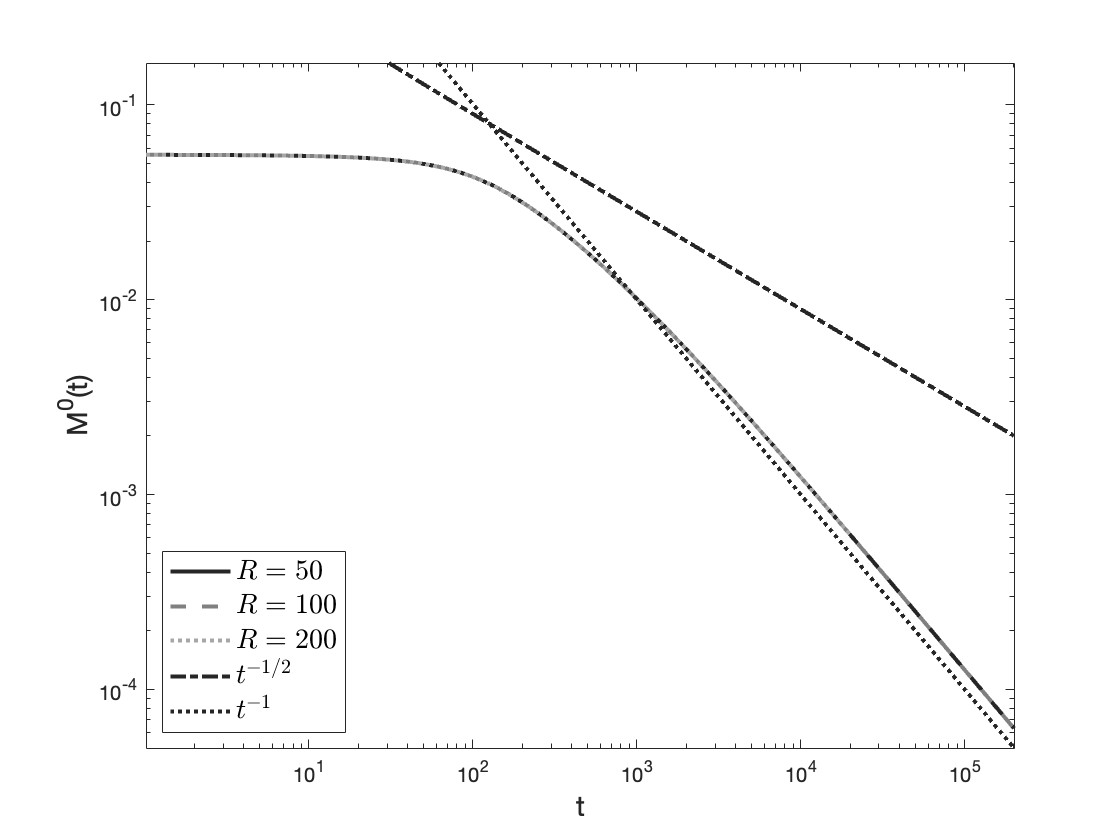}
  \caption{}
  \label{g1_decay_R}
	 \end{subfigure}
	\begin{subfigure}[b]{0.5\linewidth}
	\includegraphics[width = \linewidth, height = 0.795\linewidth]{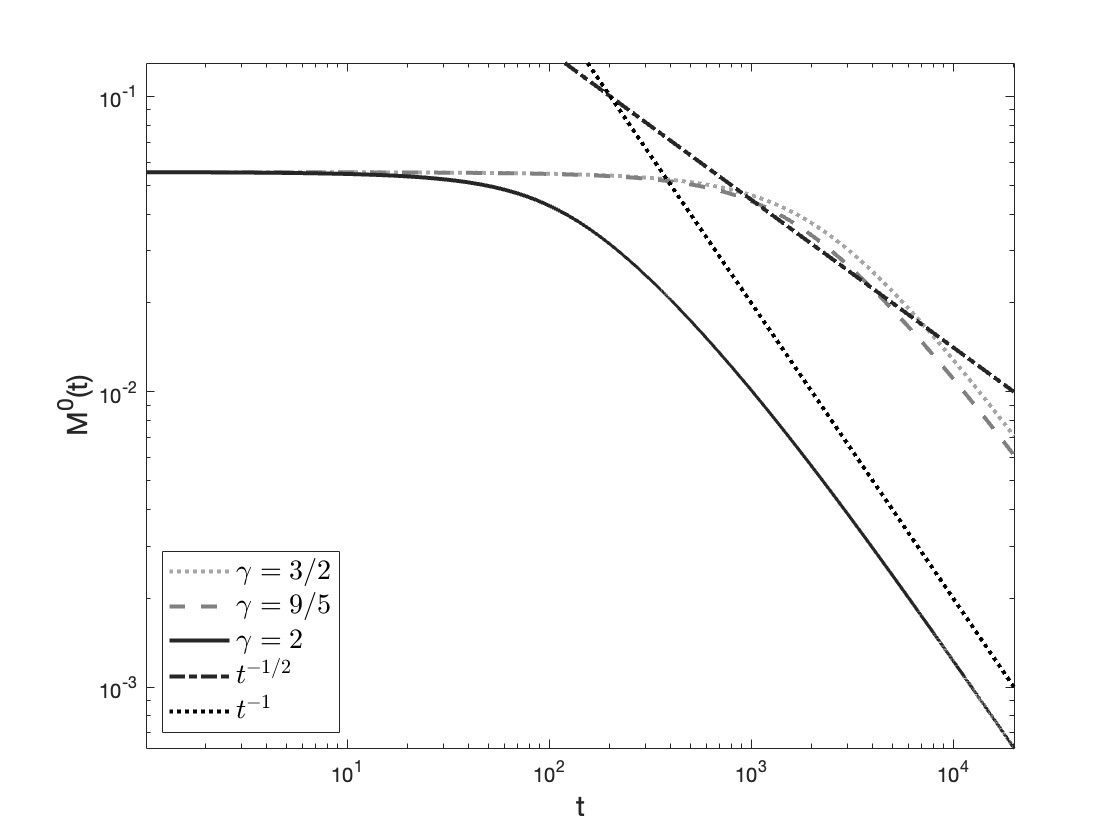}
 \caption{}
 \label{rate_decay_test_1}
	\end{subfigure}
	\caption{\small(a) Rate of decay of the total energy, corresponding to initial condition \eqref{spike} with degree $\gamma = 2$, allowing $R$ to vary. The theoretical cascade rate is shown for comparison. (b) Rate of decay of the total energy corresponding to the initial condition \eqref{spike} with theoretical cascade rate plotted for comparison.}
 \label{}
	\end{figure}
	
	\par
	We present below some preliminary observations of the so-called transient spectra, which are predicted to occur for finite capacity systems.  These spectra are different from the KZ spectra discussed above and cannot be estimated from dimensional arguments or via applying Zakharov transformations in the 3-WKE.  Briefly, the transient spectra should occur just before the singular behavior time $t^*_1$, and the solution should have the form $f_{trans}(t,k)\approx C_{trans}k^{-a}$ for $a > 0$, where, importantly, $a\neq \kappa$ for $\kappa$ the exponent of the KZ solution.  In Figure \ref{g1_trans_spectra}, we give a log-log plot of the initial condition and the solution just before the cascade process begins.  We draw a comparative line through this snapshot in time of the solution.   Lines corresponding to the KZ spectra for capillary and acoustic waves are also shown for comparison. To rigorously compute $a$, one must solve a nonlinear eigenvalue problem as in \cite{connaughton2010dynamical}.  However, as discussed above, solving such a nonlinear  eigenvalue problem is a difficult task:  in \cite{connaughton2010dynamical},    a hypothesis is imposed on the evolution of the solution. The solution is  assumed to grow linearly in time \eqref{ColmEnergy2} and  additional hypotheses are also imposed to treat the singularities of the integral \eqref{ColmEnergy3}.  While these hypotheses remain interesting mathematical questions to be verified, we  assume the nonlinear interactions to be solely responsible for the transfer of energy in our work.  We provide Figure \ref{g1_trans_spectra} as preliminary evidence that the solution seems to give a cascade behavior consistent with the predicted transient spectra, though further rigorous analysis is required for verification.  {\it To reiterate, our main concern in the present article is the behavior of $g(t,k)$ after the first blow-up time and to show the existence of the multiple blow-up time phenomenon $t_1^*<\cdots<t_n^*<\cdots$, whereas the transient cascade occurs just prior to the first blow-up time $t_1^*$. Finding self-similar profiles for the solutions before the $n$-th blow-up times, similarly to those done in \cite{bell2017self,connaughton2010dynamical,semisalov2021numerical},  is non-trivial and will be the subject of our coming work. }
	
	\begin{figure}
		\begin{center}
        \includegraphics[scale = 0.18]{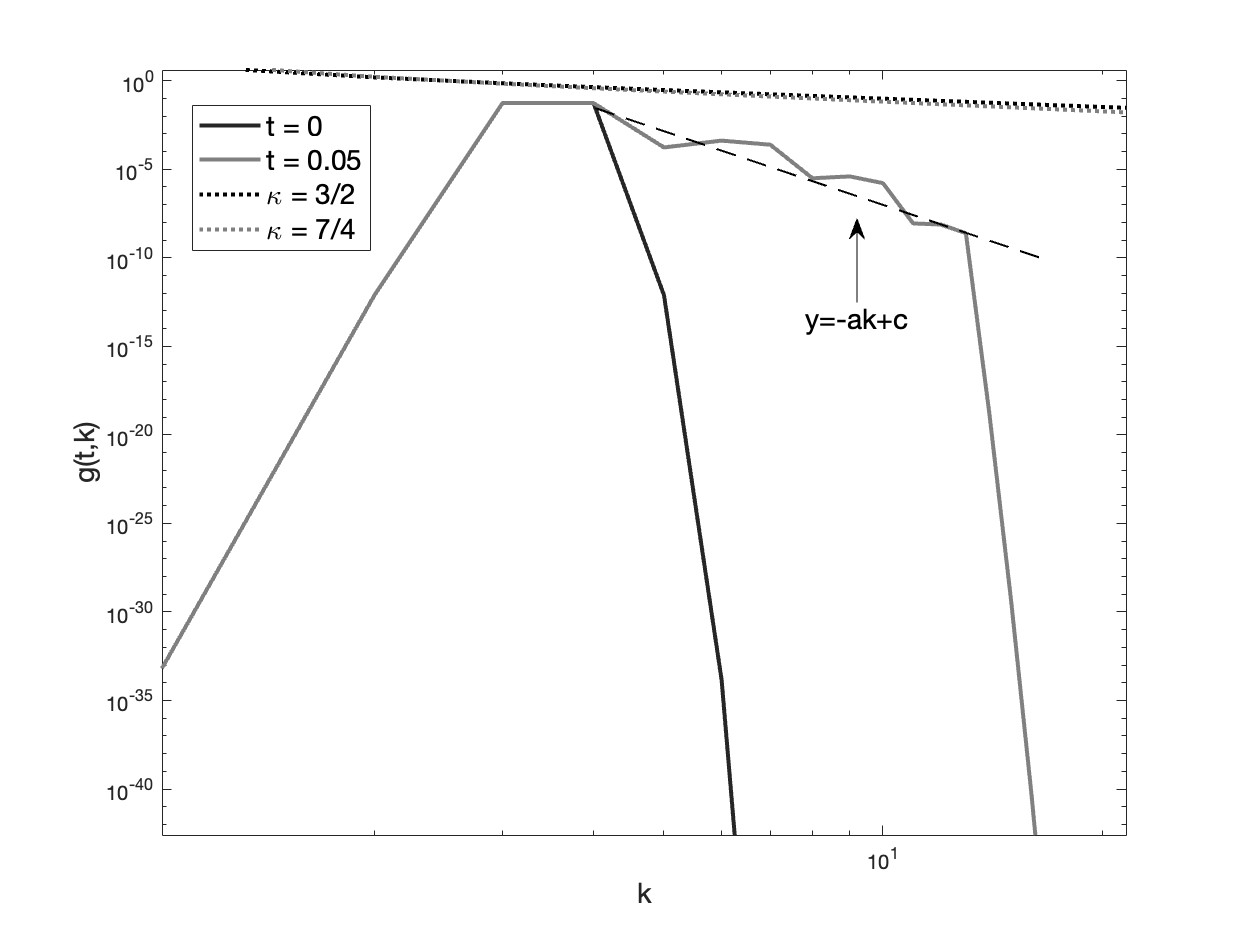}
			\caption{\small Transient cascade. }
	\label{g1_trans_spectra}
		\end{center}
	\end{figure}

	
	\subsection{Test 2}\label{test 2}
	We next choose a Gaussian further away from the origin
	\begin{equation}\label{gauss}
		g_0(k) = (5\pi)^{-1/2}e^{-(k-50/3)^2/2.5}
	\end{equation}
	as our initial condition.  Solutions are computed up to $T=10000$ with $\Delta t = 0.005$ and $h = 0.5$.  The initial condition and final state can be seen in Figure \ref{gauss_init} and \ref{gauss_final}, respectively.  This test is designed based on Test 1, as we are curious to see what happens if we move the Gaussian away from zero.
	
\begin{figure}
\centering
\begin{subfigure}[b]{.45\linewidth}
\includegraphics[width=\linewidth]{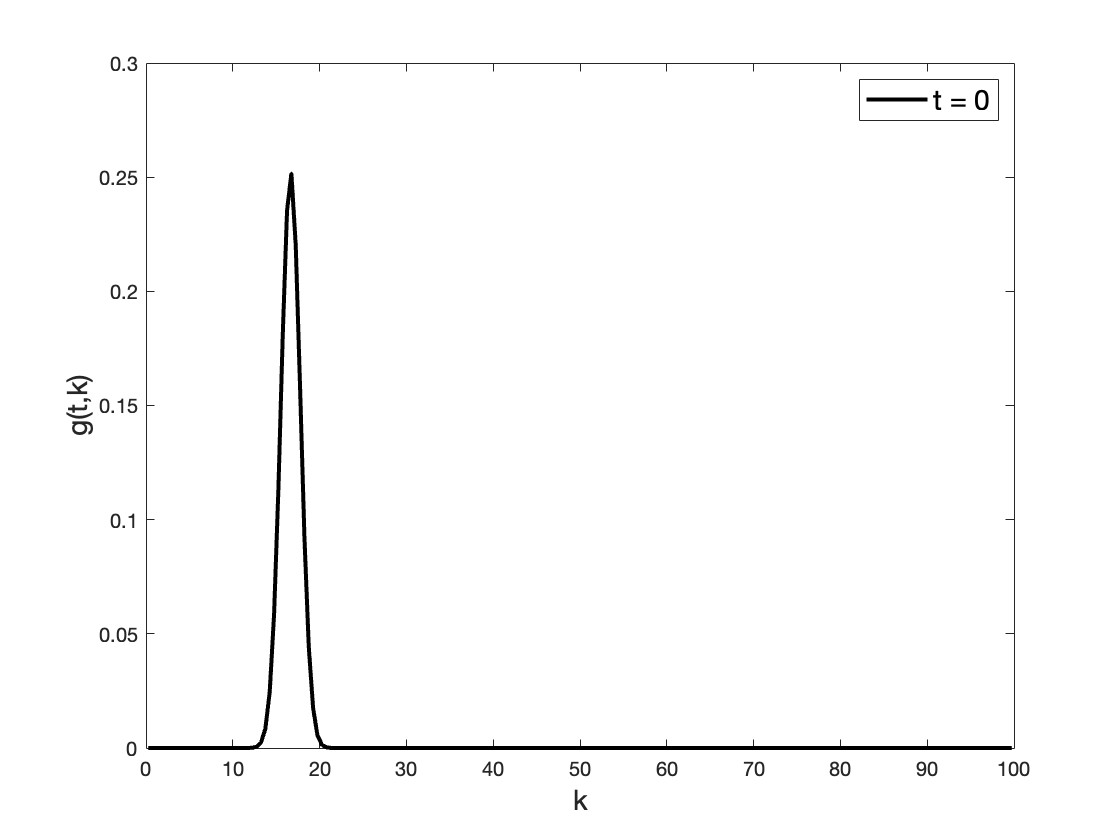}
\caption{\small Initial Condition}\label{gauss_init}
\end{subfigure}
\begin{subfigure}[b]{.45\linewidth}
\includegraphics[width=\linewidth]{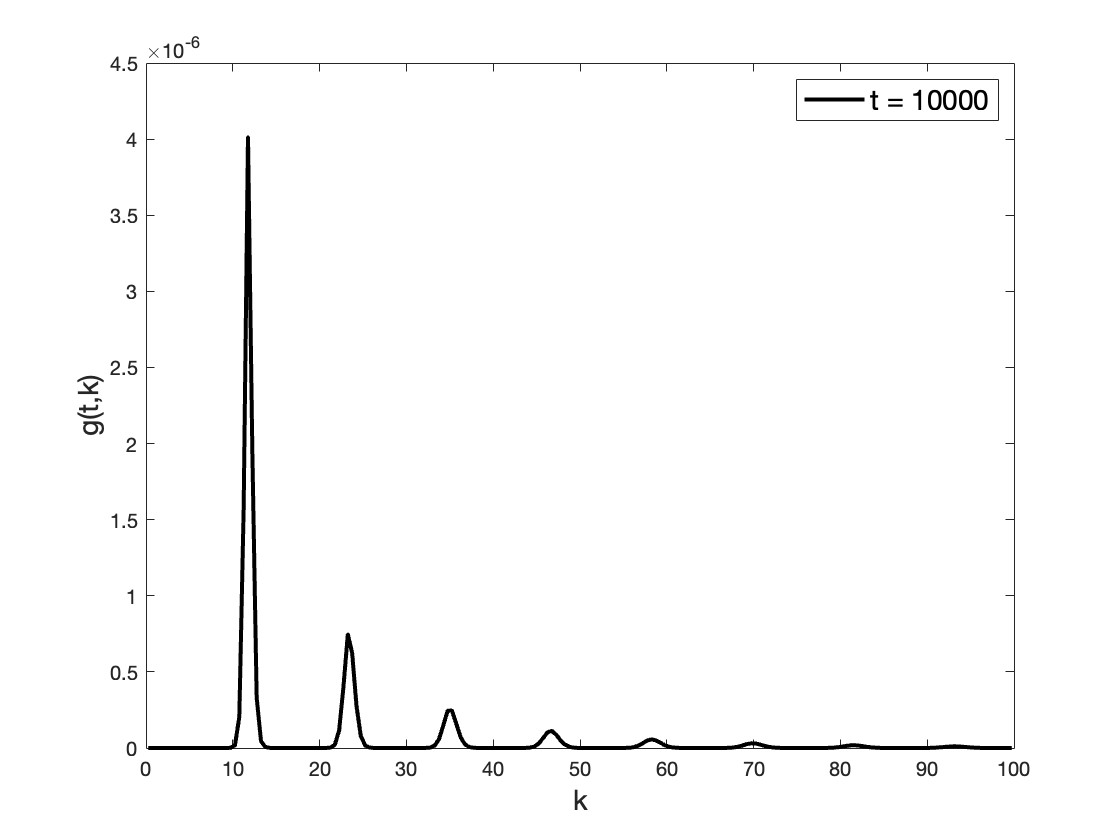}
\caption{\small Final Condition}\label{gauss_final}
\end{subfigure}
\caption{\small Test Case 2}
\label{gauss_fig}
\end{figure}

	In figures \ref{gauss_init} and \ref{gauss_final} we see that the energy is pushed slightly toward the origin at some $T_s \in [0, T)$ to $k=11.75$, away from its initial concentration at $k=16.667$ at $t=0$.  From this time $T_s$ onward, the $L^\infty$ norm of $\|g(t,k)\chi_{[0,R]}(k)\|_{L^\infty}$ decreases to $4.01\times 10^{-5}$ for $t=10000$, maintaining its concentration at $k=11.75$. This indicates  the energy cascade phenomenon does happen and happens in a very special way.  An analysis then needs to be done to explain this long time behavior of the solution. \begin{figure}
		\begin{center}
    	\includegraphics[scale = 0.3]{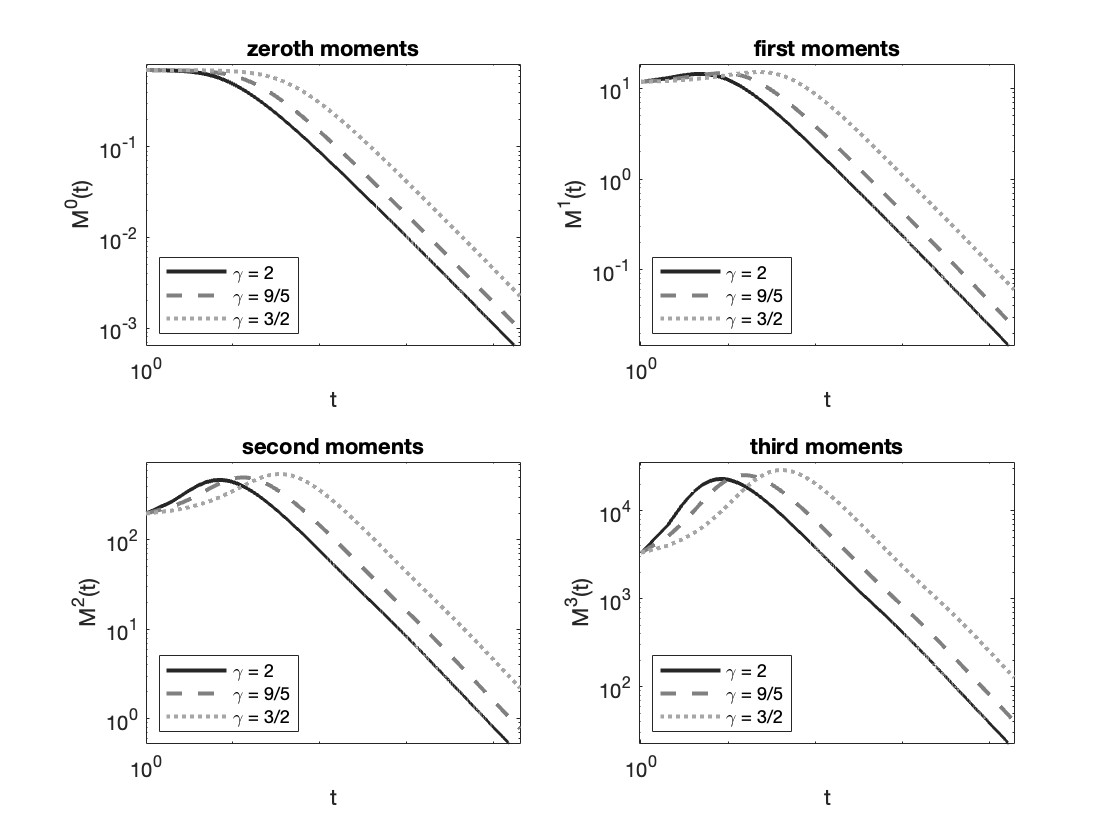}
			\caption{\small Here, we fix $R=100$ and compute the moments of the energy as a function of time for $\gamma = 3/2, 9/5, 2$ with initial condition (\ref{spike}).}
			\label{gauss_moments_gamma}
		\end{center}
	\end{figure}
	\par 
	In Figure \ref{gauss_moments_gamma}, we give results for the computation of the first few moments of the energy when allowing the degree to vary while holding the truncation value fixed.  We notice a similar behavior as in the previous test case.  
	\par
	The moment calculations are performed with $\gamma = 2$ fixed and varying truncation parameter.  The results are plotted in Figure \ref{gauss_moments_R}. In contrast to Figure \ref{g1_moments_R}, the difference in moments is more distinguishable.  However, as already mentioned, this is consistent with the previous analysis in \cite{soffer2019energy}, being that the chance of finding energy in a larger frequency interval is higher.
	
	\begin{figure}
		\begin{center}
		\includegraphics[scale = 0.3]{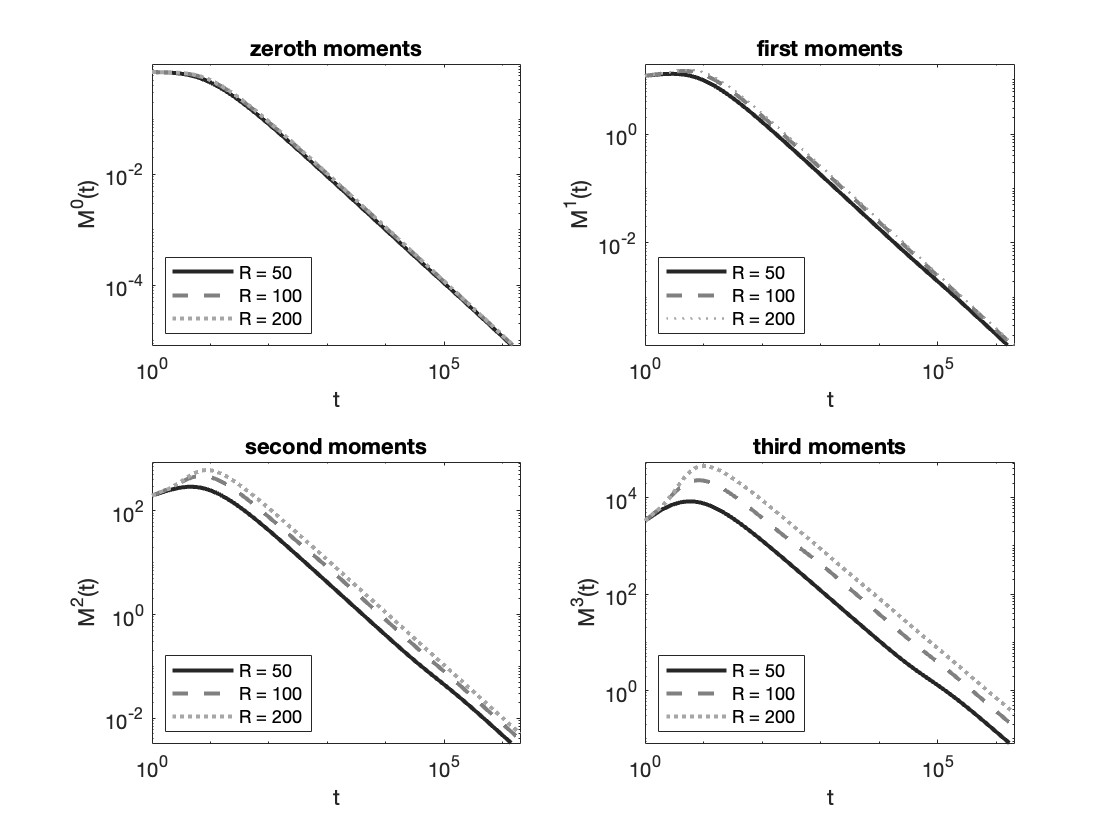}
   \caption{\small Moments of $g(t,k)$ with initial condition (\ref{gauss}) and fixed degree, $\gamma = 2$, and varying truncation parameter.}\label{gauss_moments_R}		
	\end{center}
	\end{figure}
	\par 
	
	The theoretical decay rate is compared with the decay of total energy for all considered values of $\gamma$ and $R=100$ in Figure \ref{decay_rate_test_2}. As in Test 1, the numerical results  have a good agreement with the theoretical findings of \cite{soffer2019energy}. It  appears that the decay is more like $\mathcal{O}(\frac{1}{t})$, which is bounded by $\mathcal{O}(\frac{1}{\sqrt{t}})$ as shown in \eqref{Decomposition4}. 
	\begin{figure}
    \centering
    	\begin{subfigure}[b]{.49\linewidth}
	\includegraphics[width = \linewidth]{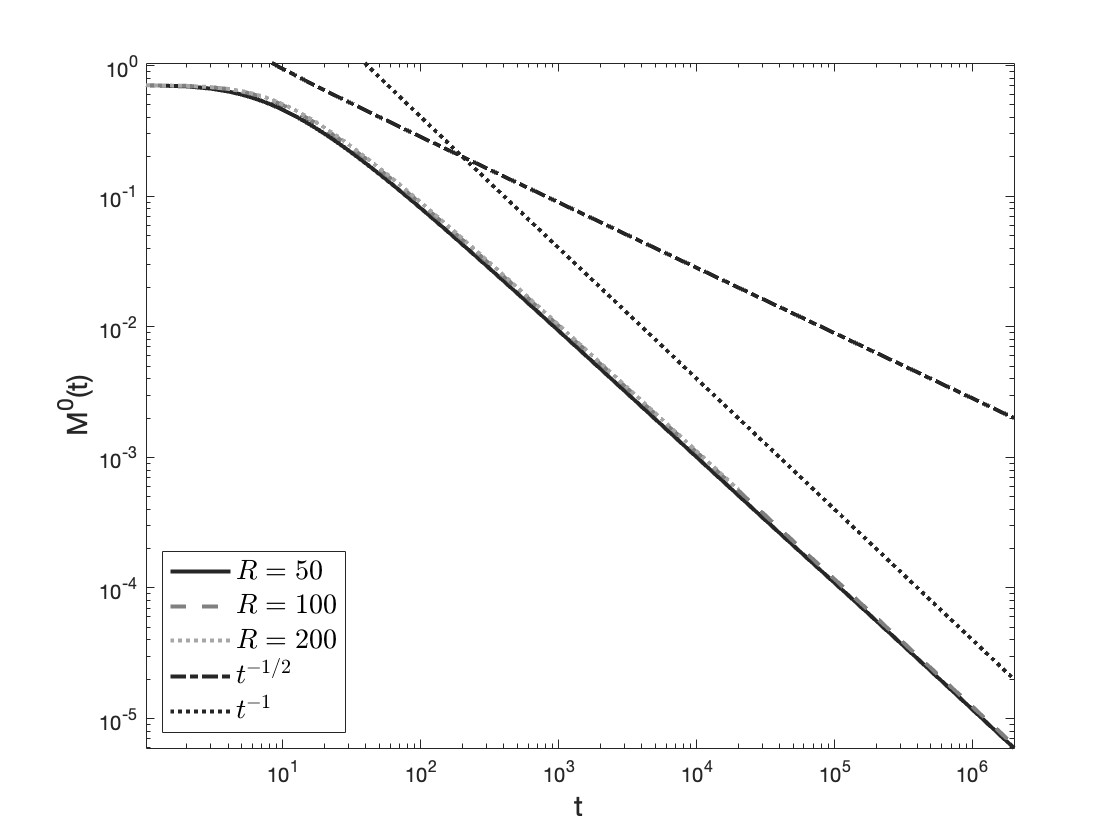}\caption{}\label{gauss_total_energy_R}
	\end{subfigure}
    \begin{subfigure}[b]{.49\linewidth}
	\includegraphics[width = \linewidth]{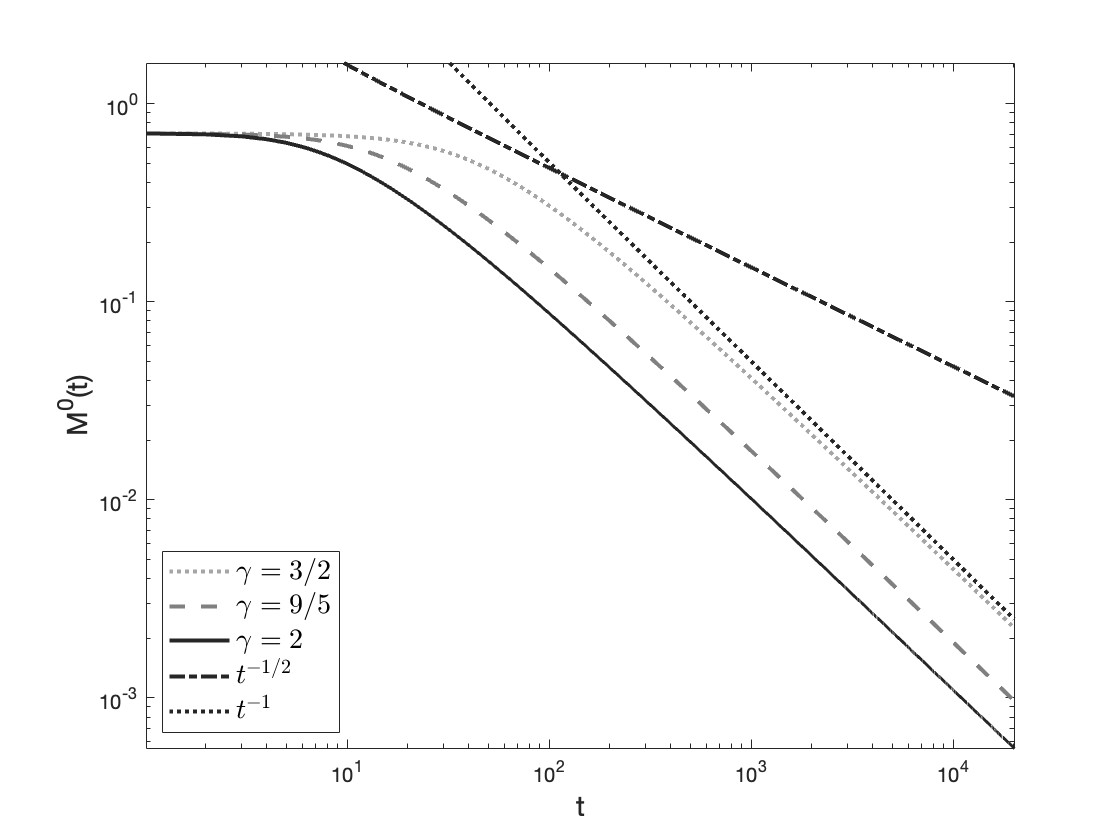}
	\caption{}\label{decay_rate_test_2}
	\end{subfigure}
	\caption{\small (a) Rate of decay of the total energy with varying truncation parameter corresponding to the initial condition shown in figure \ref{gauss_init}. (b) Rate of decay of the total energy corresponding to the initial condition shown in figure \ref{gauss_init} with varying degree and theoretical rate plotted for comparison.}
	\end{figure}
	We also compare the theoretical cascade rate to the decay of total energy for each interval considered in Figure \ref{gauss_total_energy_R}. 

	We see that increasing the truncation parameter has no influence on the cascade rate and blow-up times, consistent with the theory found in \cite{soffer2019energy}.  We notice that here and in the previous test case, that there is also not a distinguishable difference (if any) in the amount of energy contained in the intervals after varying the truncation parameter, which will contrast with the next two test cases, where much more energy is available initially.

As in the previous test, we provide possible evidence for the transient spectra, and compare with the known KZ spectra of the relevant systems.  The results are shown in Figure \ref{gauss_trans_spectra}.
	
	\begin{figure}
		\begin{center}
        \includegraphics[scale = 0.18]{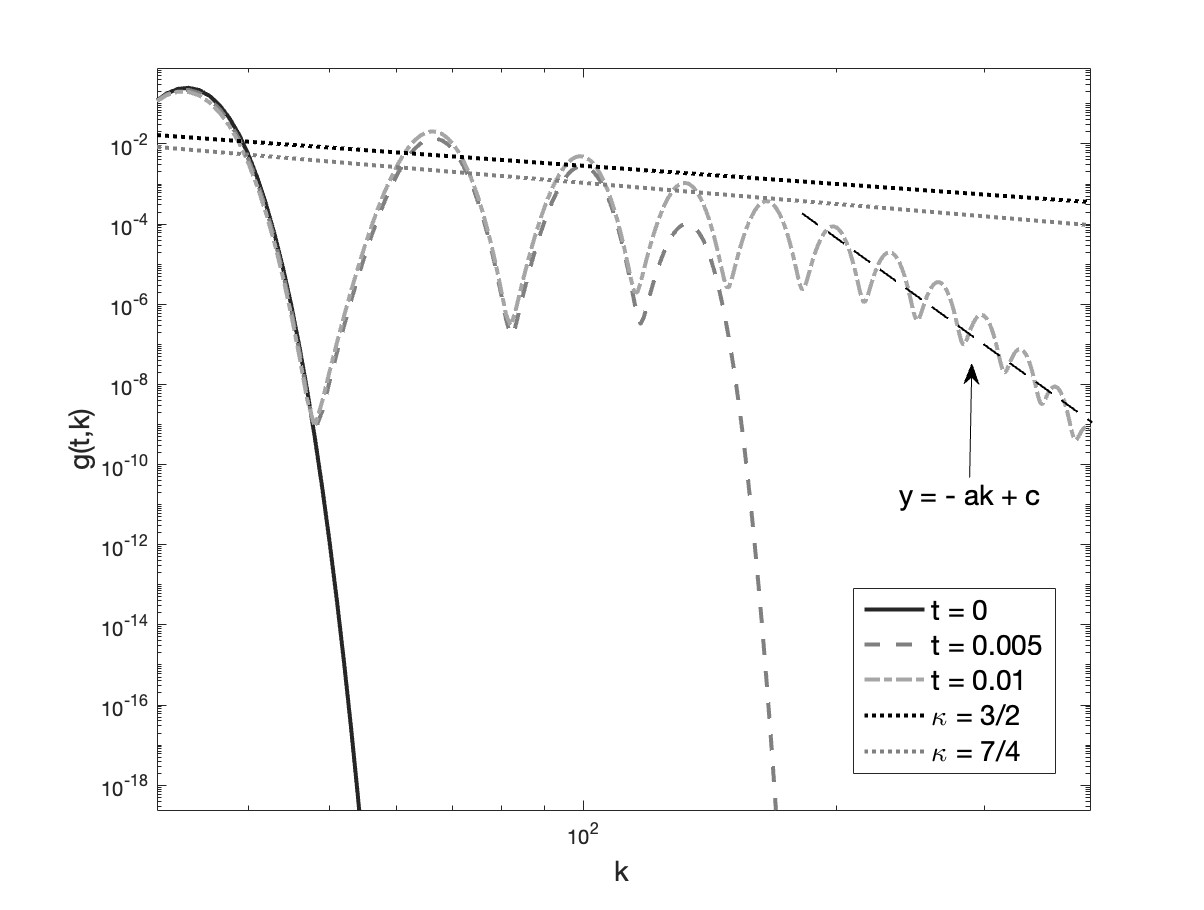}
			\caption{\small Transient cascade corresponding to the initial condition \ref{gauss}. }
	\label{gauss_trans_spectra}
		\end{center}
	\end{figure}

	
	\subsection{Test 3}\label{test3}
	Here, we consider  initial data given by 
	\begin{equation}\label{test_3_init}
		\begin{aligned}
			g_0(k) = 
			\begin{cases}
				1 & k\in \big[2n\pi, (2n+1)\pi\big]\\
				0 & k\in \big((2n+1)\pi , 2(n+1)\pi\big) \\
			\end{cases} & \text{ for } n =0,1,3,5,\ldots
		\end{aligned}
	\end{equation}
	and perform test for $t\in [0, T]$ for $T=100$ and $\Delta t = 0.0004$ when $R=25$ and $R = 50$ and $\Delta t = 0.00025$ for $R= 80$.  The frequency step is $h = 0.1$ for each interval $[0,R]$ considered. 
	
	\par
	
	In Figure \ref{square_init} we show the initial condition and in Figure \ref{square_final} the final state at $T=100$. In the final state, it would appear that the remaining energy in the interval is collected near $k=0$ with decreasing maximum amplitude at $k=0.05$. It seems that this profile is maintained as $T\rightarrow \infty$, in a similar fashion to Tests 1 and 2.  
	
\begin{figure}
\centering
\begin{subfigure}[b]{.45\linewidth}
\includegraphics[width=\linewidth]{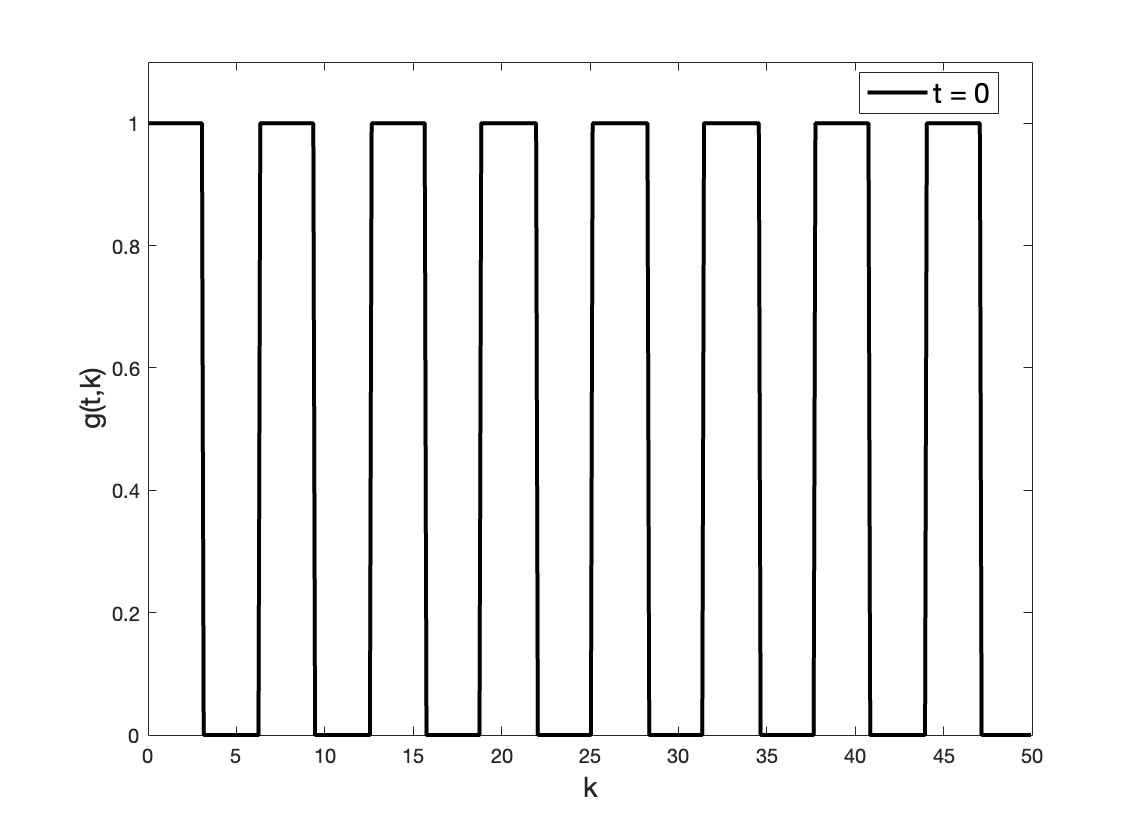}
\caption{\small Initial Condition}\label{square_init}
\end{subfigure}
\begin{subfigure}[b]{.45\linewidth}
\includegraphics[width=\linewidth]{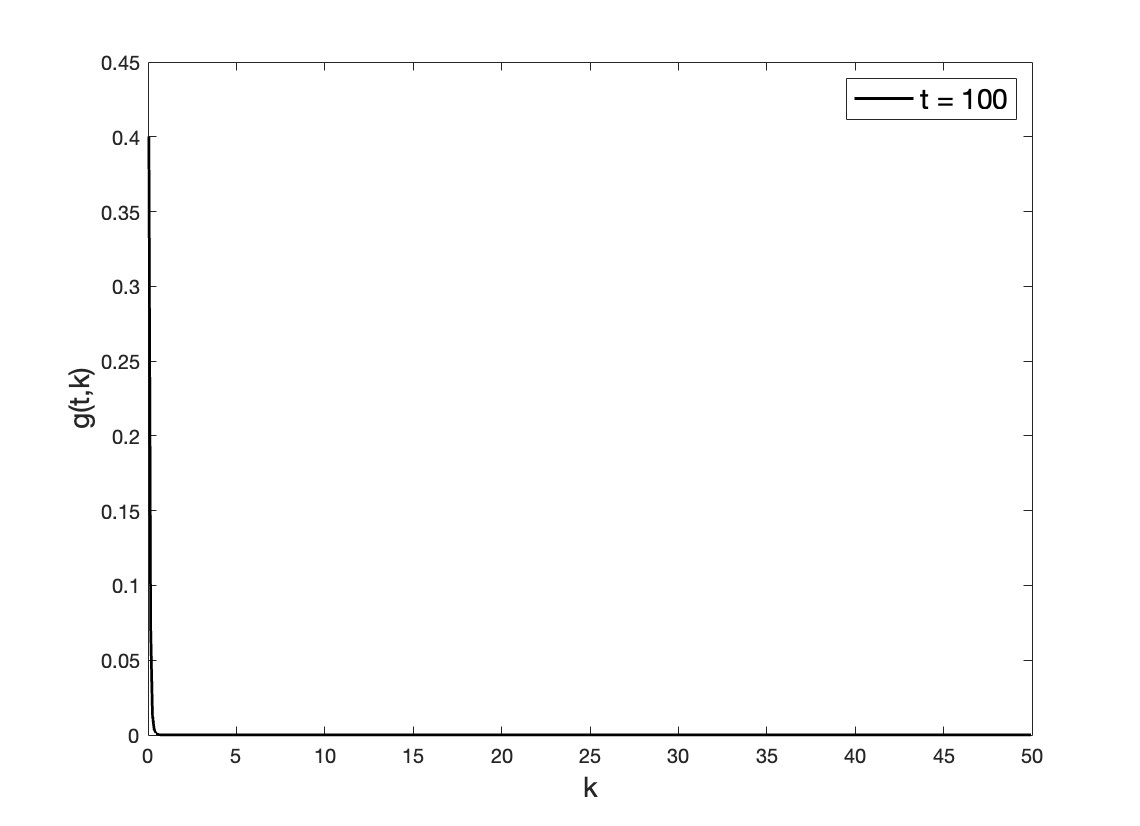}
\caption{\small Final Condition}\label{square_final}
\end{subfigure}
\caption{\small Test Case 3}
\label{square}
\end{figure}
	
	\par

	We then vary the truncation parameter and show the decay of the energy for $R = 25, 50, 80$ in Figure \ref{square_moments_trunc_param}.  As mentioned previously, we can now see that the amount of energy increases with the interval size, but, importantly, the rate of decay is the same for each truncation value.

	As in the previous tests, we explore varying the degree as seen in Figure \ref{square_decay}.  Here, as in the next case, we begin to see further contrasting behaviour as compared with the previous two test cases.  We see that the smaller the value of the degree, the longer the energy is conserved, but now, once the first blow-up time $t_1^*$ is reached, the rate of decay is larger for $\gamma = 3/2, 9/5$ as compared to $\gamma = 2$. To see this, we add a reference line corresponding to a rate of decay like $\mathcal{O}(\frac{1}{t})$ in Figures \ref{square_moments_trunc_param} and \ref{square_decay}.  It would appear that for smaller values of $\gamma$, the decay rate is better described by $\mathcal{O}(\frac{1}{t})$, but for $\gamma = 2$, the theoretical bound  $\mathcal{O}(\frac{1}{\sqrt{t}})$ (see \eqref{Decomposition4}), provides an excellent description of the cascade rate.
\begin{figure}
\centering
    \begin{subfigure}[b]{.49\linewidth}
    \includegraphics[width =\linewidth]{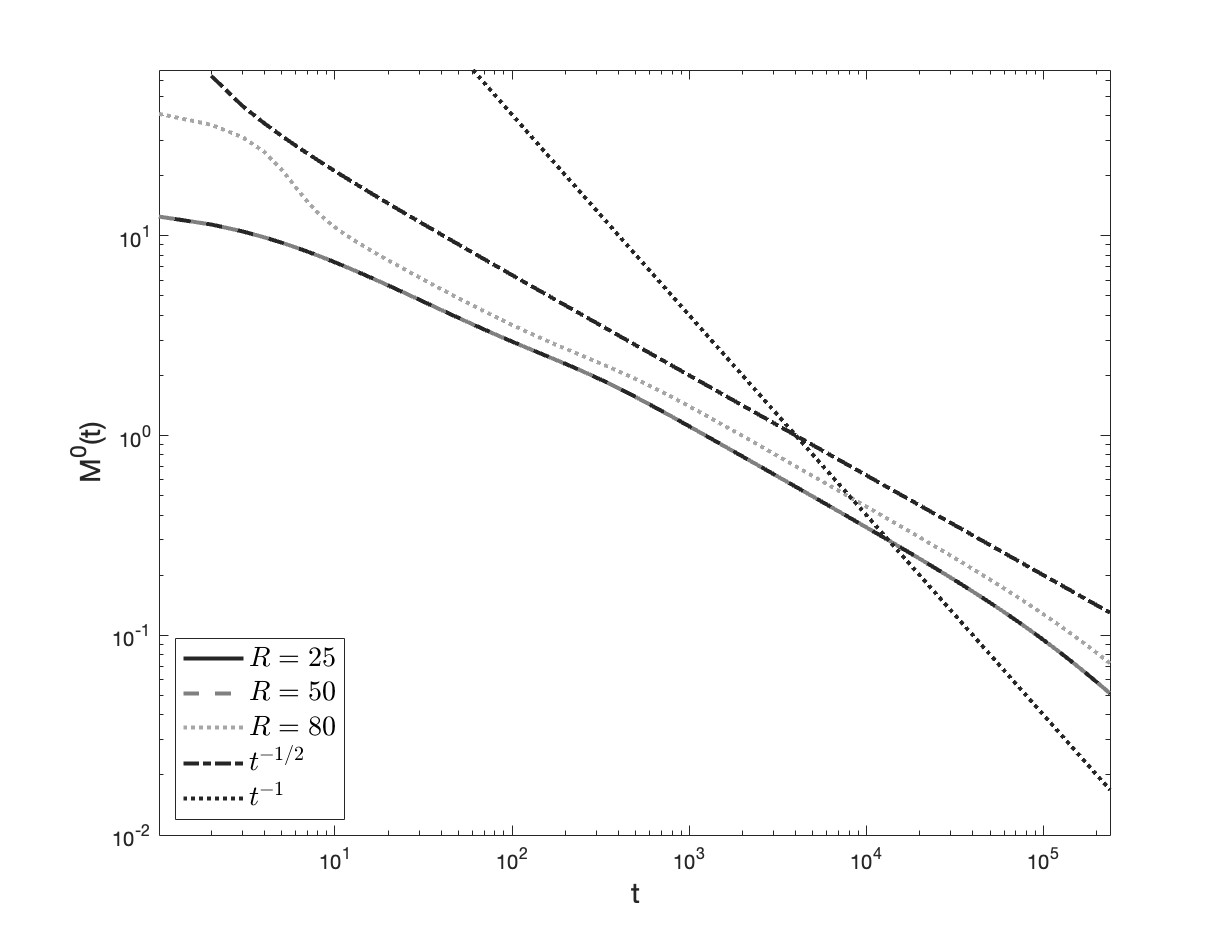}
 \caption{}\label{square_moments_trunc_param}
\end{subfigure}
	\begin{subfigure}[b]{.49\textwidth}	\includegraphics[width = \linewidth]{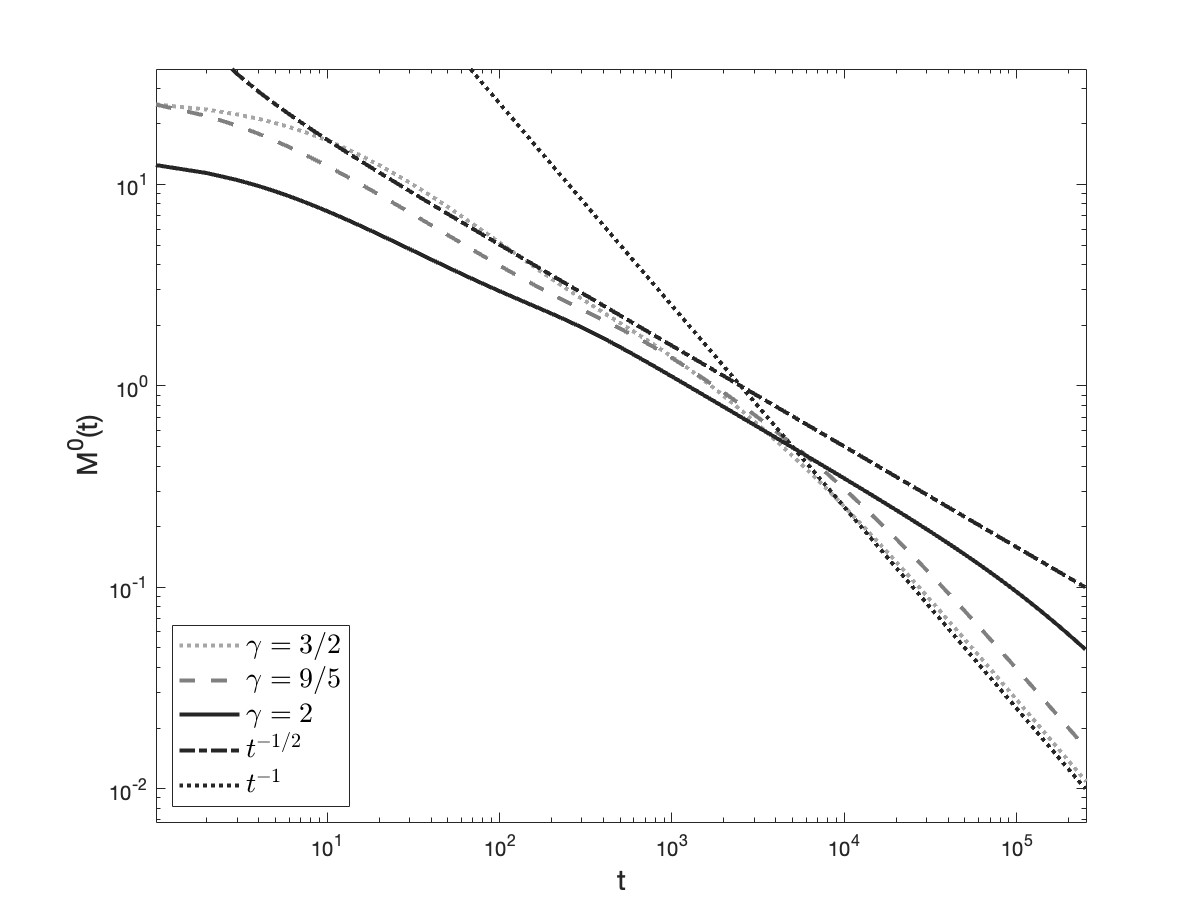}
	\caption{}\label{square_decay}
	\end{subfigure}
	\caption{\small (a) Zeroth moments of solution corresponding to initial condition \eqref{test_3_init}, with $\gamma = 2$ and allowing $R$ to vary. (b) Decay rate for initial condition \eqref{test_3_init} plotted against the theoretical rate.  The red and black lines are translated to match the intersection of the moments.}
\end{figure}

	
	\subsection{Test 4}\label{test4}
	Our last test has initial data given by
	\begin{equation}\label{test4_init}
		\begin{aligned}
			g_0(k) = 
			\frac{k - 2n\pi}{2\pi} && k\in [2n\pi, 2(n+1)\pi),
		\end{aligned}
	\end{equation}
	for $n\in \mathbb{N}_0$. As in Test 3, we set $h = 0.1$, $T = 100$ and $\Delta t = 0.0004$ when $R=25$ and $R = 50$ but $\Delta t = 0.00025$ for $R= 80$.
	\par
	We again give initial and final conditions in Figures \ref{saw_init} and \ref{saw_final}, respectively.  As in Test 3, we observe that the energy is accumulated to a frequency nearer to  $k=0$ in some finite time $T_s$, where it remains fixed with decaying $L^\infty$ norm. 
\begin{figure}
\centering
\begin{subfigure}[b]{.45\linewidth}
\includegraphics[width=\linewidth]{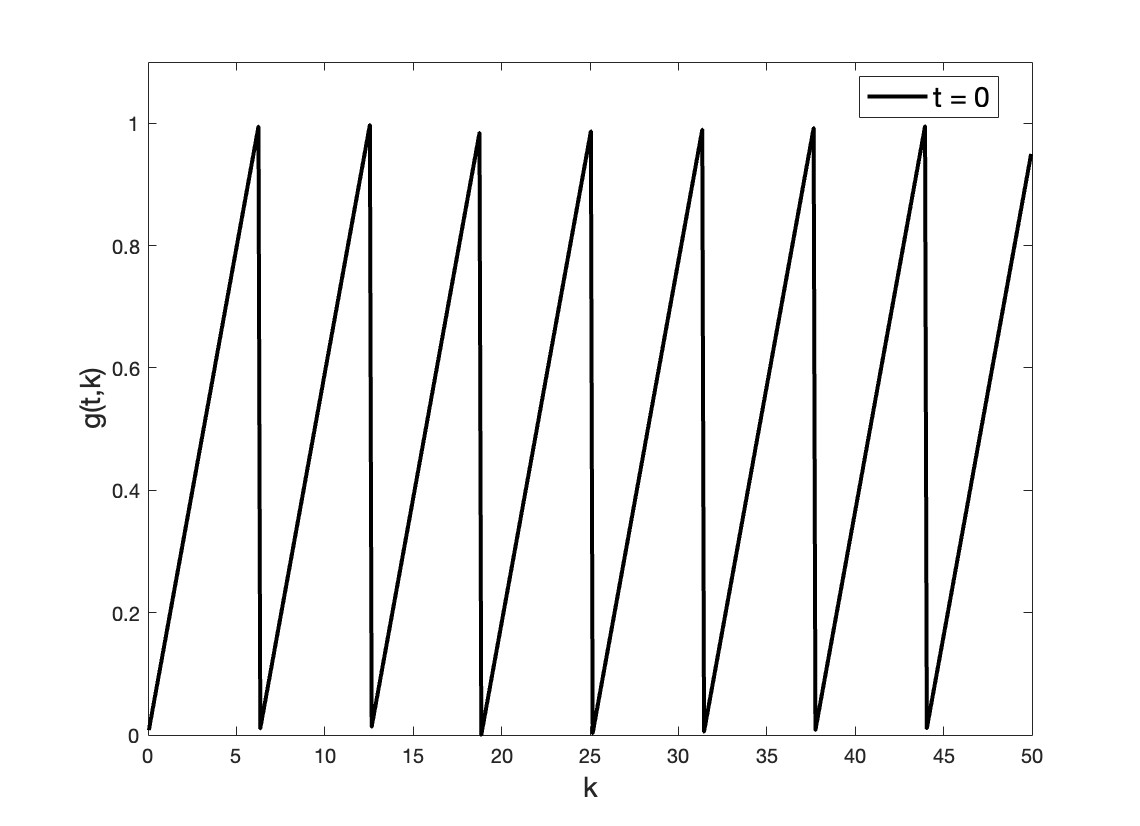}
\caption{\small Initial Condition}\label{saw_init}
\end{subfigure}
\begin{subfigure}[b]{.45\linewidth}
\includegraphics[width=\linewidth]{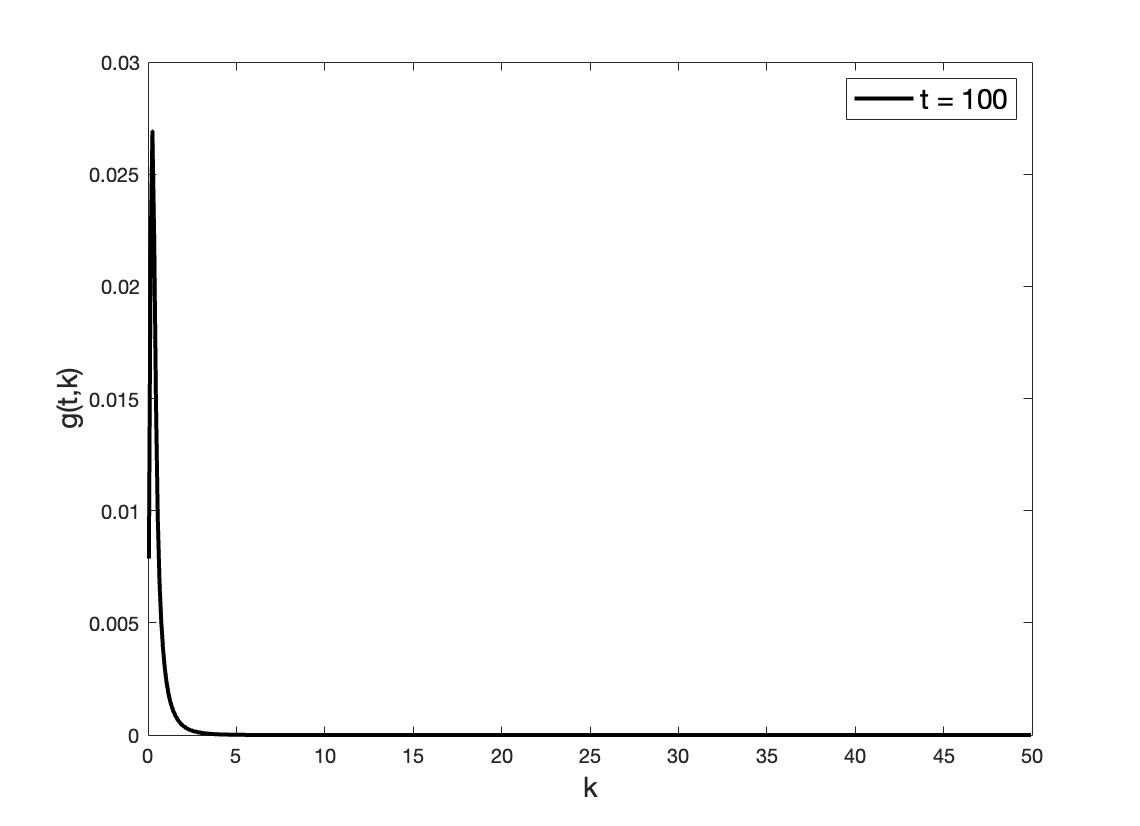}
\caption{\small Final Condition}\label{saw_final}
\end{subfigure}
\caption{\small Test Case 4}
\label{saw}
\end{figure}

	\par
	
	Now fixing $\gamma = 2$ once again, we see in Figure \ref{saw_moments_trunc_param} that varying the truncation parameter has a similar result on the total energy as in the previous test case.  By extending the interval, we see that a larger amount of energy is retained, but that the rate of decay is equal for all truncation values.  Further, we see a good fit with the theoretical decay rate.
	\begin{figure}
	\centering
	\begin{subfigure}[b]{.49\linewidth}
		\includegraphics[width = \linewidth]{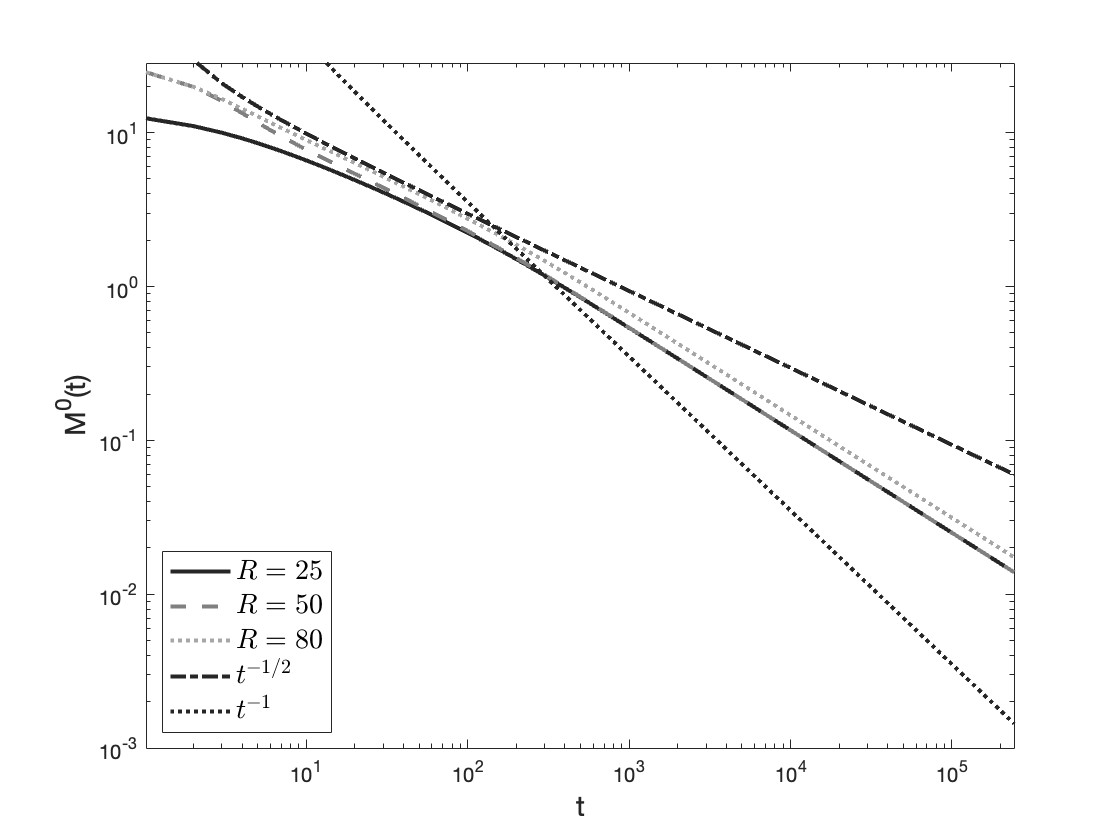}
			\caption{}	\label{saw_moments_trunc_param}
	\end{subfigure}
	\begin{subfigure}[b]{.49\linewidth}
	\includegraphics[width=\linewidth]{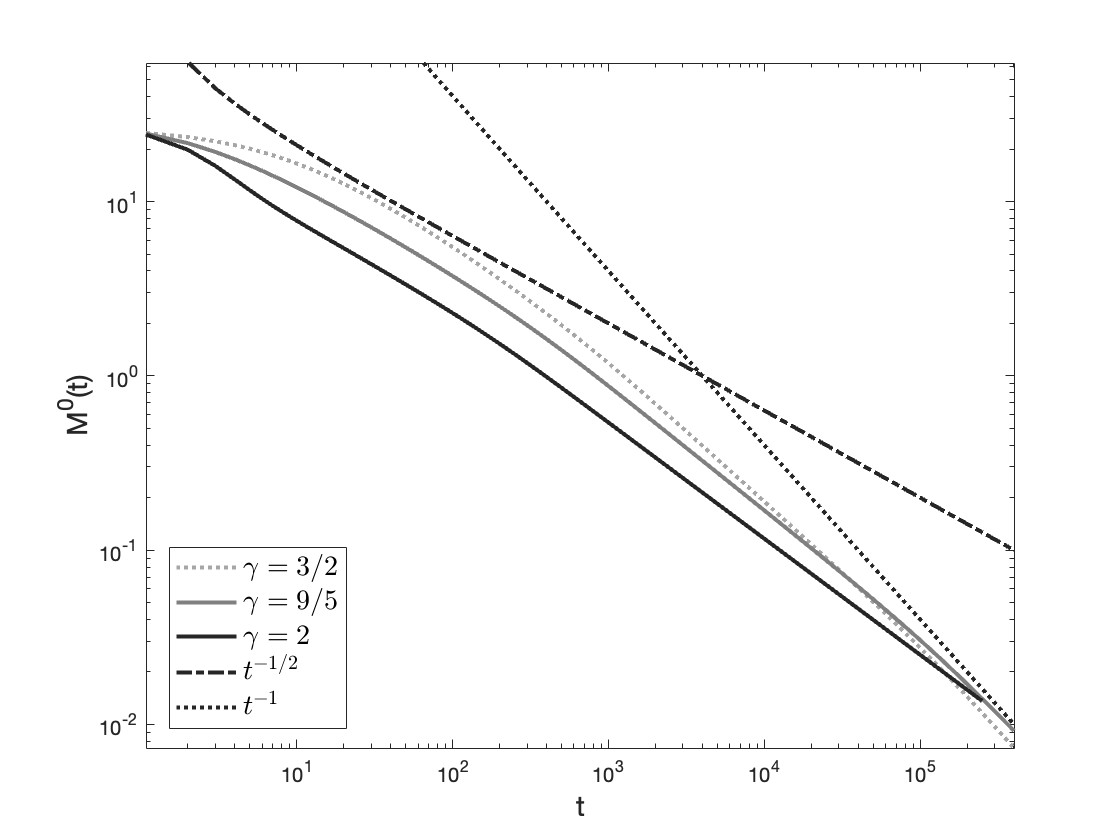}
			\caption{\small }
		\label{saw_moments_degree}
	\end{subfigure}
	\caption{(a) Zeroth moments of solution corresponding to initial condition \eqref{test4_init}, with $\gamma = 2$ and allowing $R$ to vary.  The theoretical decay rate is shown for comparison. (b) Zeroth moments of solution corresponding to initial condition \eqref{test4_init}, with $R = 50$ and allowing $\gamma$ to vary. The theoretical decay rate is shown for comparison.}
	\end{figure}
	\par

	\par
	Next, we let $\gamma =\frac{3}{2}, \frac{9}{5}, 2$ and keep $R = 50$ once again, and show the zeroth moments of these solutions in Figure \ref{saw_moments_degree}. Here, as for the previous test cases, for $\gamma = 3/2, 9/5$, we see that the energy is conserved for a longer amount of time when compared to $\gamma = 2$. As for the previous initial condition, it would appear that once the decay begins, the rates of decay are faster for $\gamma = 3/2, 9/5$ than for $\gamma = 2$.  We again observe that for $\gamma = 2$, the decay rate is more like $\mathcal{O}(\frac{1}{\sqrt{t}})$ and for the smaller values of $\gamma$, the decay rate is more like $\mathcal{O}(\frac{1}{t})$.

	\section{Conclusions and Further Discussion}\label{conc}

	We introduce a finite volume scheme which allows us to observe the long time asymptotics of the solutions of isotropic 3-wave kinetic equations, including the energy cascade behavior proved in \cite{soffer2019energy}. Our numerical algorithm is based on the combination of the identity represented in Lemma \ref{lemma:identity} and   Filbet and Lauren\c cot's scheme \cite{Fil04}   for the Smoluchowski coagulation equation.
	\par
	From the four numerical tests, we can see that the energy cascade behavior happens for $\gamma = \frac{3}{2}, \frac{9}{5}, 2$, and seem to verify the theory found in \cite{soffer2019energy}.  
	
	\par	
	
	From the solutions computed in sections \ref{test 1}, \ref{test 2}, \ref{test3}, and \ref{test4}, One can see that the smaller $\gamma$ is, the slower is the onset of decay. The energy cascade behavior also seems to occur independently of the smoothness of the initial data for all four cases with $\gamma=\frac{3}{2}, \frac{9}{5}, 2$. 
	
	\par 
	

The results in figure \ref{g1_moments_R} and \ref{gauss_moments_R} serve as another verification of the theory since the decay rate of the energy in any finite interval is the same due to the following fact proved in the main theorem of \cite{soffer2019energy}:
\begin{equation*}
 \int_{0}^R g(t,\omega)d\omega\ = \	\int_{\mathbb{R}_+}\chi_{[0,R]}(\omega)   g(t,\omega)d\omega\le \  \mathcal{O}\Big(\frac{1}{\sqrt{t}}\Big) \mbox{   as  } t\to\infty,
 \end{equation*}
 for all truncation parameters  R.
 
Therefore, the rate that the energy leaves any finite interval $[0,R]$ is the same since the convergence does not depend on the   truncation parameter.
As the theory is only for the energy cascade at the point $\omega=\infty$, it is a challenging task theoretically to obtain a convergence study pointwise with respect to all of  the velocity  variables in this $\omega$.

We would like to comment that in contrast to Tests 3 and 4, the amount of energy contained in the interval looks indistinguishable for the various truncation parameters in Tests 1 and 2.  This can be explained by comparing $\| g^0 \|_{L^1(0,R)}$ of the different test cases.  The initial energy serves as a kind of reservoir in the finite intervals \cite{Nazarenko:2011:WT}.  Then, for tests like Tests 3 and 4, we increase the amount of initial energy by increasing the size of the interval.  That is for  $R_1 < R_2 < R_3$ we have
\begin{equation}
    \| g^0 \|_{L^1(0,R_1)} < \| g^0 \|_{L^1(0,R_2)} < \| g^0 \|_{L^1(0,R3)}.
\end{equation}
This is not the case (or is negligible) for Tests 1 and 2, seen in figures \ref{g1_moments_R}, \ref{gauss_moments_R}, due to the initial conditions selected there.  The important thing to notice is that the slopes (decay rates) are independent of how much energy is contained in the interval initially, though the amount of energy can vary depending on the initial condition. 

The numerical results confirm the theoretical bound \eqref{Decomposition4} and show that the decay rate should be  $\mathcal{O}\big(\frac{1}{t^s}\big)$, with $s\in[\frac12, 1]$ for various initial data. In section \ref{test3}, the cascade rate is  described quite well by $\mathcal O\big(\frac{1}{\sqrt t}\big)$ for $\gamma = 2$. We then conclude that the cascade rate bound obtained in \cite{soffer2019energy} is optimal.

 
\section*{Ackowledgements}	
The authors would like to thank Prof. T. Hagstrom 
	for the use of SMU’s high-performance computing cluster ManeFrame II (M2), Prof. B. Rumpf for his many crucial suggestions and Prof. S. Nazarenko for several constructive suggestions to improve  numerical plots. The Los Alamos unlimited release number is LA-UR-23-20290.

\bibliographystyle{siamplain}

\bibliography{WTbib_short}
	
\end{document}